\documentclass[11pt,reqno]{amsart}

\usepackage[margin=1.1in]{geometry}
\usepackage{microtype}                     
\usepackage{subcaption}
\usepackage{placeins}

\usepackage{amsmath,amssymb,amsfonts,amsthm} 
\usepackage{mathrsfs,euscript}             
\usepackage{mathtools}
\usepackage{stmaryrd}
\SetSymbolFont{stmry}{bold}{U}{stmry}{m}{n}
\usepackage{graphicx}                      
\usepackage[dvipsnames]{xcolor}            
\colorlet{blue}{black}                     
\usepackage{tikz}                          
\usetikzlibrary{shapes.geometric, arrows, positioning}

\usepackage{hyperref}

\theoremstyle{plain}
\newtheorem{theorem}{Theorem}[section]
\newtheorem{lemma}[theorem]{Lemma}
\newtheorem{proposition}[theorem]{Proposition}
\newtheorem{corollary}[theorem]{Corollary}

\theoremstyle{definition}

\newtheorem{remark}[theorem]{Remark}

\numberwithin{equation}{section}

\newcommand{\paren}[1]{\left(#1\right)}
\newcommand{\jump}[1]{\llbracket#1\rrbracket}

\newcommand{\p}{\partial}

\newcommand{\at}[2]{\left. #1 \right|_{#2}}
\newcommand{\grad}[1]{\nabla #1}

\newcommand{\mc}[1]{\mathcal{#1}}
\newcommand{\wh}[1]{\widehat{#1}}
\newcommand{\wt}[1]{\widetilde{#1}}
\newcommand{\bm}[1]{\boldsymbol{#1}}
\newcommand{\abs}[1]{\left\lvert #1 \right\rvert}
\newcommand{\norm}[1]{\left\lVert #1 \right\rVert}

\newcommand{\R}{\mathbb{R}}

\newcommand{\Corr}{\mathcal C}

\usepackage{siunitx}
\usepackage{booktabs}
\usepackage[ruled,vlined,linesnumbered]{algorithm2e}
\SetKwBlock{Begin}{begin}{end}
\SetKw{Continue}{continue}
\SetKw{Break}{break}
\SetKwRepeat{Do}{do}{while}
\SetKwInput{Input}{Input}
\SetKwInput{Output}{Output}
\SetKwInput{Parameters}{Parameters}

\title[Cartesian grid method for moving-domain advection-diffusion equations]{A Cartesian Grid Method for Advection-Diffusion Equations with Robin Boundary Conditions on Moving Domains}

\author{Han Zhou}
\address{Department of Mathematics, University of Pennsylvania, Philadelphia, PA 19104, USA}
\email{hzhou24@sas.upenn.edu}

\author{Yoichiro Mori}
\address{Department of Mathematics, Department of Biology, University of Pennsylvania, Philadelphia, PA 19104, USA}
\email{y1mori@sas.upenn.edu}

\author{Lingxing Yao}
\address{Department of Mathematics, University of Akron, Akron, OH 44325, USA}
\email{lyao@uakron.edu}



\begin{document}

\begin{abstract}
We develop a Cartesian grid method for advection--diffusion equations with Robin boundary conditions on moving domains.
The moving-domain problem is reformulated as an interface problem on a box, with an unknown density introduced on the moving interface to enforce the Robin condition.
The bulk equation is discretized by a cell-centered finite-difference scheme on the Cartesian grid, while interface corrections are obtained from local problems in a narrow band around the interface.
The resulting method requires only modest computational geometry, avoids remeshing and cut cells, and is compatible with geometric multigrid and matrix-free GMRES.
The GMRES iteration count is essentially independent of the mesh size, and the computational cost scales linearly with the number of bulk degrees of freedom.
For the one-dimensional scheme, first-order convergence in time and second-order convergence in space are proved.
Numerical examples in one and two dimensions, including manufactured solutions and an active transport problem without an exact solution, demonstrate the accuracy and efficiency of the method.
\end{abstract}

\maketitle


\section{Introduction}
Advection--diffusion equations on moving domains arise when bulk transport is coupled to flux exchange across a moving boundary. A motivating example comes from osmotic and electrodiffusive descriptions of cell migration and cell-scale water flow. In confined geometries, water permeation can drive tumor-cell migration even when cytoskeletal and motor machinery is disrupted~\cite{Stroka2014}; related theoretical models show how osmotic pumping, ionic fluxes, and membrane water fluxes can generate motion or volume changes~\cite{AtzbergerIsaacsonPeskin2009,Mori2011,LiMoriSun2015}. The osmotic-flow problem studied by Yao and Mori~\cite{Yao2017} is particularly relevant: an osmotically active solute is transported in a deforming membrane-bounded region, and passive permeability and active pumps determine transmembrane solute fluxes. 

Computationally, such problems present three intertwined difficulties: discretizing a PDE on a complex domain whose boundary changes in time, accurately capturing a sharp moving interface with possibly complicated boundary conditions, and solving the resulting discrete systems efficiently over many time steps. Barrett et al. developed a hybrid semi-Lagrangian cut-cell method for advection--diffusion equations with Robin boundary conditions on complex, time-dependent domains~\cite{Barrett2022}; the method is conservative and directly resolves the boundary flux, but it requires cut-cell geometry and a treatment of small-cell effects. Fryklund et al. developed an integral-equation method for advection--diffusion on moving planar domains~\cite{Fryklund2023}, obtaining high accuracy through elliptic marching, special quadrature, extension procedures, and fast evaluation of modified Helmholtz potentials. Other relevant approaches include embedded-boundary methods for heat and Poisson equations on irregular domains~\cite{McCorquodaleColellaJohansen2001,SchwartzBaradColellaLigocki2006}, Robin discretizations for heat, Stefan-type, and Poisson-type problems on irregular domains~\cite{PapacGibouRatsch2010,AriasBochkovGibou2018,BochkovGibou2019,ChaiEtAl2020,ChaiEtAl2021}, and sharp-interface Cartesian methods for moving-boundary flow computations~\cite{Udaykumar2001,MarellaKrishnanLiuUdaykumar2005}.

The method developed here introduces an unknown interface density and designs correction terms at near-interface cells so that the finite-difference operator in the bulk remains the standard one. This correction-based viewpoint goes back in part to Mayo's method, which combined Cartesian-grid fast solvers with integral-equation formulations and correction terms near irregular boundaries to solve elliptic problems without body-fitted meshes~\cite{Mayo1984,Mayo1985}. Immersed interface methods and augmented immersed interface methods use jump conditions, and sometimes additional interface unknowns, to correct finite-difference stencils near singular sources, discontinuous coefficients, or irregular boundaries~\cite{LeVequeLi1994,Li2003,LiIto2006,LiItoLai2007}. Correction-function methods construct a local representation of the jump in the solution and its derivatives in a narrow band, and then use that representation to correct irregular stencils while keeping the regular-grid operator unchanged~\cite{MarquesNaveRosalesSeibold2011,WiegmannBube2000}. Kernel-free boundary-integral methods replace integral evaluation by equivalent interface problems on Cartesian grids and use local corrections with fast PDE solvers~\cite{YingHenriquez2007,zhou2024correction}. The present work extends this correction-based framework to a time-dependent moving-domain problem with Robin boundary conditions.

The numerical analysis of sharp-interface Cartesian grid methods is still much less complete than their computational use. Several works by Beale and collaborators are especially close to the present study: they prove uniform second-order accuracy for correction-based immersed interface discretizations~\cite{BealeLayton2006}, maximum-norm smoothing estimates for localized parabolic errors~\cite{Beale2009Smoothing}, and nearly second-order uniform error estimates for Navier--Stokes flow with an exact moving boundary~\cite{Beale2015NS}. The augmented-variable setting is harder, because the stability of the coupled system becomes more intricate; the finite-difference analogue of the boundary-integral method in~\cite{BealeYing2019} addresses this issue based on the discrete potential theory on Cartesian grids. Other rigorous convergence analyses include immersed-interface, piecewise-polynomial interface, parabolic mixed-boundary, and moving-interface finite element schemes~\cite{HuangLi1999,ChenStrain2008,BouchonPeichl2010Mixed,Guo2021}. These results are highly relevant, but they do not directly cover augmented-variable Cartesian grid methods in which local corrections impose Robin boundary conditions on moving domains. 

In this paper we develop and analyze a Cartesian grid method for advection--diffusion equations with Robin boundary conditions on moving domains. We rewrite the moving-domain problem as an interface problem on a larger box domain and introduce an interface unknown, which plays the role of a parabolic single-layer density. The bulk equation is discretized by a cell-centered finite-difference scheme with backward Euler time stepping, while local polynomial correction functions in a narrow space--time neighborhood of the moving boundary impose the jump information needed to enforce the Robin condition. These corrections enter only on the right-hand side at irregular grid nodes, so the bulk coefficient matrix remains the standard Cartesian finite-difference matrix and can be inverted by geometric multigrid. This is a central distinction from cut-cell and embedded-boundary finite-volume approaches: no cut-cell volumes, partial-cell quadrature, grid--interface intersection reconstruction, or small-cell stabilization is required. 

We also prove convergence for a fully discrete one-dimensional version of the scheme. The analysis proceeds by establishing a discrete maximum principle for the scheme. We demonstrate that the large consistency errors that arise from discretizing the boundary condition and from the moving domain do not affect the overall accuracy because of their localization in space and/or time. Although the analysis is restricted to one space dimension, it appears to be the first convergence proof for a correction-based Cartesian grid method for a moving-domain parabolic problem of this type.

The remainder of the paper is organized as follows. Section~\ref{sec:model} describes the moving-domain advection--diffusion problem. Section~\ref{sec:method} presents the reformulated interface problem, corrected finite-difference scheme, including the boundary-Cauchy-type problem for the correction function and the interface discretization. Section~\ref{sec:numerical-analysis} establishes the error estimates for the fully discrete scheme in one dimension. Section~\ref{sec:results} reports several numerical examples, and Section~\ref{sec:conclusion} gives concluding remarks.

\section{Model Problem}\label{sec:model}
Let $\mc B\subset \R^2$ be a fixed rectangular domain. For each $t\in[0,T]$, let $\Gamma(t)\subset\mc B$ be a smooth closed interface such that $\Gamma(t)\cap\p\mc B=\emptyset$, and let $\Omega(t)\subset\mc B$ denote the (time-dependent) region enclosed by $\Gamma(t)$. Let $\bm X(s,t):\mathbb S^1\to \Gamma(t)$ be a parametrization of the interface, where $\mathbb S^1$ denotes the periodic one-dimensional parameter domain. We consider the evolution of a chemical concentration $c(\bm x,t)$ on the moving domain $\Omega(t)$, driven by diffusion and advection by a prescribed background velocity field. We define the positive-time space--time domain, lateral boundary, and exterior space--time region by
\begin{equation}
  Q_T = \cup_{t\in(0,T]}\Omega(t)\times\{t\},\quad \Sigma_T = \cup_{t\in(0,T]}\Gamma(t)\times\{t\},\quad Q_T^c = (\mc B\times(0,T])\setminus \overline{Q_T}.
\end{equation}
Let $\bm u=(u,v)$ be the background fluid velocity. The concentration $c$ satisfies the advection--diffusion equation
\begin{equation}
  \p_t c + \grad\cdot\paren{\bm u c - D\grad c} = f, \quad \text{in } Q_T,
\end{equation}
where $D>0$ is a constant diffusion coefficient, together with the initial condition
\begin{equation}
  c(\bm x, 0) = c_0(\bm x), \quad \bm x \in \Omega(0).
\end{equation}
Let $\bm n$ denote the outward unit normal on $\p\Omega(t)=\Gamma(t)$. On $\Sigma_T$, we impose a Robin-type flux condition,
\begin{equation}
  \paren{(\bm u - \p_t\bm X) c - D\grad c} \cdot \bm n = g_\Gamma, \quad\text{ on }\Sigma_T.
\end{equation}
Here, $f$ and $g_\Gamma$ are prescribed bulk and boundary source terms defined in $Q_T$ and $\Sigma_T$, respectively.
If $\Omega(t)$ is instead taken to be the exterior of $\Gamma(t)$ (within $\mc B$), then additional boundary conditions are required on $\p\mc B$. For example, in the case of cell migration in a tube, we may impose periodicity in the horizontal direction and homogeneous Neumann conditions $\p_{\bm n}c=0$ on the top and bottom walls.

We assume that the interface motion is prescribed by a normal velocity law
\begin{equation}
  \p_t \bm X(s,t) \cdot \bm n\big(\bm X(s,t),t\big) = V_n\big(\bm X(s,t),t\big), \quad (s,t)\in \mathbb S^1\times [0,T],
\end{equation}
where $V_n(\bm x,t)$ is a given normal velocity field. In more realistic settings, the interface velocity may be coupled to the concentration field, for example,
\begin{equation}
  \p_t \bm X(s,t) = \bm V\big(\bm X(s,t),c(\bm X(s,t),t),t\big), \quad (s,t)\in \mathbb S^1\times [0,T],
\end{equation}
where $\bm V$ is a nonlinear function/operator of both $\bm X$ and $c$.

Throughout the formulation section we assume that the interface, data, and
prescribed motion are sufficiently regular on the time interval considered.
The algorithm developed below only requires the interface position and normal
velocity at each time level, so it can be used with prescribed motion or with a
separately updated coupled motion law.

\section{Numerical method}\label{sec:method}
\subsection{Extended equations}
We now reformulate the moving-domain IBVP as an interface problem on the fixed
box $\mc B$. This form keeps the bulk discretization on a Cartesian grid while
encoding the Robin boundary condition through an interface density. From this
point on, all variables are dimensionless: the diffusion coefficient is scaled to $D=1$, and the
physical flux condition in Section~\ref{sec:model} becomes
$\partial_{\bm n}c+\alpha c=g$, where
$\alpha=(\partial_t\bm X-\bm u)\cdot\bm n/D$ and
$g=-g_\Gamma/D$ after the corresponding rescaling of time, velocity, and source
terms. We first write the formulation for zero initial data.
Thus we consider the Robin IBVP
\begin{equation}\label{eqn:adv-diff-2}
\begin{aligned}
    \partial_t c + \grad\cdot(\bm u c - \grad c) &= f , \quad&&\text{ in }Q_T, \\
    \p_{\bm n}c + \alpha c  &= g, \quad&&\text{ on }\Sigma_T,\\
    c &= 0, \quad&& \text{ in }\Omega(0),
\end{aligned}
\end{equation}
where $\bm u$ and $\alpha$ are the dimensionless velocity and transfer rate.
For non-zero initial data, for example $c(\bm x,0) = c_0(\bm x)$, using the superposition principle, we first extend $c_0$ to $\mc B$ arbitrarily to obtain an extended function $\wt c_0$ and subtract the corresponding terms from $g$ and $f$ to obtain a similar equation for $d(\bm x,t) = c(\bm x, t) - \wt c_0(\bm x)$ with modified $f$ and $g$, i.e.,
\begin{equation}
\begin{aligned}
    \partial_t d + \grad\cdot(\bm u d- \grad d) &= f - \grad\cdot(\bm u \wt c_0- \grad \wt c_0), \quad&&\text{ in }Q_T, \\
    \p_{\bm n}d + \alpha d &= g - \paren{\partial_{\bm n}\wt c_0 + \alpha \wt c_0}, \quad&&\text{ on }\Sigma_T,\\
    d &= 0, \quad&& \text{ in }\Omega(0).
\end{aligned}
\end{equation}

For a piecewise smooth quantity $w$ on $\mc B$, we use the one-sided trace
notation
\begin{equation}
  \at{w}{\Gamma_i}(\bm x,t)
  = \lim_{\epsilon\downarrow0} w(\bm x-\epsilon\bm n,t),
  \qquad
  \at{w}{\Gamma_e}(\bm x,t)
  = \lim_{\epsilon\downarrow0} w(\bm x+\epsilon\bm n,t),
  \qquad \bm x\in\Gamma(t).
\end{equation}
Thus $\at{w}{\Gamma_i}$ is the
interior trace from the physical domain $\Omega(t)$ and $\at{w}{\Gamma_e}$ is
the exterior trace from $\mc B\setminus\overline{\Omega(t)}$. The jump and
average are $\jump{w}=\at{w}{\Gamma_i}-\at{w}{\Gamma_e}$ and
$\{w\}=(\at{w}{\Gamma_i}+\at{w}{\Gamma_e})/2$.
We extend $f$ to $\mc B\times(0,T]$ by setting $f=0$ in $Q_T^c$ and seek a
function on the full box satisfying
\begin{equation}\label{eqn:adv-diff-inf}
\begin{aligned}
    \partial_t c + \grad\cdot(\bm u c - \grad c) &=  f, \quad&&\text{ in }\mc B\times (0,T], \\
    \jump{c} &= 0, \quad&&\text{ on }\Sigma_T,\\
    \jump{\partial_{\bm n} c} &= \psi, \quad&&\text{ on }\Sigma_T,\\
    c &= 0, \quad&& \text{ in }\mc B \times\{t=0\}, \\
    c &= 0, \quad && \text{ on }\partial\mc B\times[0,T].
\end{aligned}
\end{equation}
Here $\psi$ is an additional unknown density on $\Sigma_T$. It is determined by the equation
\begin{equation}\label{eqn:adv-diff-bie}
     \frac{1}{2}\psi +\{\p_{\bm n}c + \alpha c\} = g, \quad\text{ on }\Sigma_T.
\end{equation}
For compatibility, one needs to set $\psi = 0$ for $t=0$.
The equivalence with the Robin IBVP is immediate from the interface conditions. On $\Sigma_T$, since $\jump{c}=0$ and $\jump{\partial_{\bm n}c}=\at{\partial_{\bm n}c}{\Gamma_i}
-\at{\partial_{\bm n}c}{\Gamma_e}=\psi$, equation~\eqref{eqn:adv-diff-bie} gives
\begin{equation}
\frac12\paren{\at{\partial_{\bm n}c}{\Gamma_i}
-\at{\partial_{\bm n}c}{\Gamma_e}}
+\frac12\paren{\at{\partial_{\bm n}c}{\Gamma_i}
+\at{\partial_{\bm n}c}{\Gamma_e}}
+\alpha\at{c}{\Gamma_i}
=\at{\partial_{\bm n}c}{\Gamma_i}+\alpha\at{c}{\Gamma_i}=g.
\end{equation}
Thus the solution to the interface problem~\eqref{eqn:adv-diff-inf} satisfies
the dimensionless Robin IBVP~\eqref{eqn:adv-diff-2} in $Q_T$.
The homogeneous Dirichlet condition on $\partial\mc B$ is only a convenient
choice for the presentation. Other standard outer boundary conditions, such as
periodic or no-flux conditions, can be handled in the same Cartesian
grid framework because the box boundary is grid-aligned.
The formulation is closely related to a single-layer boundary integral formulation. In that view, $\psi$ is the unknown density, and equation~\eqref{eqn:adv-diff-bie} is the boundary integral equation for this density; see \cite{Costabel1990}.

\subsection{Corrected finite difference scheme}
We first describe the numerical discretization of the interface problem~\eqref{eqn:adv-diff-inf} for a given $\psi$.
For convenience, we assume the rectangular domain is a square $\mc B = (0,L)^2$.
For a positive integer $N$, we partition $\mc B$ by grid lines
$x_i=i h$ and $y_j=j h$, $0\le i,j\le N$, with $h=L/N$. The unknowns are stored
at cell centers
\begin{equation}
\mc T_h=\{\bm x_{i+\frac12,j+\frac12}
=(x_{i+\frac12},y_{j+\frac12}) : 0\le i,j\le N-1\}.
\end{equation}
Thus $c_{i+\frac12,j+\frac12}^n$ denotes the numerical solution at a cell
center at time $t^n$. We also discretize the interface $\Gamma(t)$ into points
$\bm X_k^n=\bm X(s_k,t^n)$, $s_k=k\Delta s$, $k=0,1,\ldots,N_\Gamma-1$, with
$\Delta s=2\pi/N_\Gamma$. We assume that this interface mesh is quasi-uniform with
neighboring point distances of order $O(h)$, after reparameterization if
needed. Let $\psi_k^n$ approximate $\psi(\bm X_k^n,t^n)$, and let $\psi^n$ be
the periodic cubic-spline interpolant determined by these values.

At a fixed time $t$, we label each cell center by its position relative to the interface. A point $\bm x\in\mc T_h$ is an interior node if $\bm x\in\overline{\Omega(t)}$, and it is an exterior node otherwise.
Let $\mathcal{S}(\bm x)\subset\mc T_h$ be the finite difference stencil at $\bm x$. We call $\bm x$ a regular node if the whole stencil lies on one side of $\Gamma(t)$. We call it an irregular node if the stencil crosses $\Gamma(t)$, that is, if it contains grid points from both $\Omega(t)$ and $\mc B\setminus\overline{\Omega(t)}$.

At a cell center $(x_{i+\frac12},y_{j+\frac12})$ away from the interface at time $t^n$, we use a standard cell-centered finite difference scheme in space and backward Euler in time:
\begin{equation}\label{eqn:fds-reg}
\mc D_t^- c_{i+\frac12,j+\frac12}^{n}
+ \mc D_x^- \left(
u_{i+1,j+\frac12}^{\,n}
\, \mc A_x^+ c_{i+\frac12,j+\frac12}^{\,n}
\right)
+ \mc D_y^- \left(
v_{i+\frac12,j+1}^{\,n}
\, \mc A_y^+ c_{i+\frac12,j+\frac12}^{\,n}
\right)
-
\Delta_h c_{i+\frac12,j+\frac12}^{\,n} = f_{i+\frac12,j+\frac12}^{\,n},
\end{equation}
where $f_{i+\frac12,j+\frac12}^{\,n}=f(x_{i+\frac12},y_{j+\frac12},t^n)$ and, for any grid function $w$, the difference operators are defined as
\begin{align}
    \mc D_t^- w_{a,b}^{n} &= \frac{w_{a,b}^{n} - w_{a,b}^{n-1}}{\tau},\\
    \mc D_x^\pm w_{a,b} &= \pm\frac{w_{a\pm 1,b} - w_{a,b}}{h}, \quad \mc D_y^\pm w_{a,b} = \pm\frac{w_{a,b\pm 1} - w_{a,b}}{h},\\
    \Delta_h w_{a,b} &= \mc D_x^+\mc D_x^- w_{a,b} +\mc D_y^+ \mc D_y^- w_{a,b}
\end{align}
and the average operators $\mc A_x,\mc A_y$ are defined as
\begin{equation}
\mc A_x^+ w_{a, b}
= \frac12\left( w_{a+1,b} + w_{a, b} \right),
\qquad
\mc A_y^+ w_{a, b}
= \frac12\left( w_{a, b+1} + w_{a, b} \right),
\end{equation}
where $a, b$ are integers or half-integers.
For the differential operator $\mc L$ with $\mc Lc = \p_t c + \grad\cdot\paren{\bm u c - \grad c}$, we abbreviate equation~\eqref{eqn:fds-reg} as
\begin{equation}
    \mc L_{\tau, h}^n c_{i+\frac12,j+\frac12}^{n} = f_{i+\frac12,j+\frac12}^{\,n}.
\end{equation}

Near the interface, a finite difference stencil may cross a jump in the extended solution. Without a correction, this produces a large local truncation error and degrades the global error. We correct the right-hand side of the discretized equation by introducing a correction function in a narrow band around the interface and adding the corresponding correction terms to~\eqref{eqn:fds-reg}.

Let $\Omega_\Gamma$ be a narrow band around the interface $\Gamma$ and denote by $\Omega_{\Sigma,T} = \cup_{t\in[0,T]}\Omega_{\Gamma(t)}\times\{t\}$.
The width of the narrow band is chosen to be sufficiently large so that all cell centers where the correction function is needed are covered, and it typically depends on the background grid and the finite difference scheme.
Let $c^+$ and $c^-$ be smooth extensions of $\at{c}{Q_T}$ and $\at{c}{Q_T^c}$ into $\Omega_{\Sigma,T}$.
We denote the exact correction function by $\Corr = c^+-c^-$. Let $\Corr_h$ be a numerical construction of the correction function, for which the numerical method will be elaborated in the next subsection. The function $\Corr$ depends on both $\psi$ and the jump of $f$.

Let $\chi_T(\bm x,t)$ be a characteristic function defined as
\begin{equation}
    \chi_T(\bm x,t) =
    \begin{cases}
        1, & (\bm x,t)\in Q_T,\\
        0, & (\bm x,t)\in Q_T^c.
    \end{cases}
\end{equation}
In the fully discrete update we write $\chi^n(\bm x)=\chi_T(\bm x,t^n)$, the characteristic function of $\Omega(t^n)$.
With the correction function and the characteristic function, we represent the solution in different domains with different expressions,
\begin{equation}\label{eqn:c-rep}
    c (\bm x,t)=
    \begin{cases}
        c^+(\bm x,t) - (1-\chi_T(\bm x,t) )\Corr(\bm x,t), &  (\bm x,t)\in Q_T\cap \Omega_{\Sigma,T}, \\
        c^-(\bm x,t) + \chi_T(\bm x,t) \Corr(\bm x,t), &  (\bm x,t)\in Q_T^c\cap \Omega_{\Sigma,T}.
    \end{cases}
\end{equation}
It is easier to approximate $\mc L c^\pm$ than $\mc L c$ directly, because $c^\pm$ are smooth across the local stencil and do not need an interface correction. Applying the finite difference operator $\mc L_{\tau, h}^n$ gives
\begin{equation}
    \mc L_{\tau, h}^n c (\bm x,t)=
    \begin{cases}
        f(\bm x,t) - \mc L_{\tau, h}^n[(1-\chi_T)\Corr](\bm x,t)  +\mc O(\tau + h^2) ,&  (\bm x,t)\in Q_T\cap \Omega_{\Sigma,T}, \\
        f(\bm x,t) + \mc L_{\tau, h}^n[ \chi_T\Corr](\bm x,t)+\mc O(\tau + h^2),& (\bm x,t)\in Q_T^c\cap \Omega_{\Sigma,T}.
    \end{cases}
\end{equation}
This motivates our corrected scheme
\begin{equation}\label{eqn:corrected-scheme}
    \mc L_{\tau, h}^n c_{i+\frac12,j+\frac12}^{n} = f_{i+\frac12,j+\frac12}^{n} + Q_{i+\frac12,j+\frac12}^{n},
\end{equation}
where the correction term $Q_{i+\frac12,j+\frac12}^{n}$ is the commutator of
the discrete operator with multiplication by the characteristic function
$\chi$:
\begin{equation}
  Q_{i+\frac12,j+\frac12}^{n}
  =
  \paren{[\mc L_{\tau, h}^n,\chi^n]\Corr^{n}_h}_{i+\frac12,j+\frac12}
  =
  \paren{\mc L_{\tau, h}^n(\chi^n \Corr^{n}_h)-\chi^n\,\mc L_{\tau, h}^n\Corr^{n}_h}_{i+\frac12,j+\frac12}.
\end{equation}
The correction function $\Corr_h$ is only evaluated near the interface. Thus the correction term $Q$ is nonzero only at irregular grid nodes.
Because the interface treatment changes only the right-hand side, the coefficient matrix remains the same. This is useful for fast solvers.

The correction function also gives a direct way to discretize the boundary condition~\eqref{eqn:adv-diff-bie}. We use Lagrange interpolation to approximate $c^\pm$ and $\partial_{\bm n}c^\pm$ on $\Gamma(t)$ from nearby grid values. Then we impose~\eqref{eqn:adv-diff-bie} at the interface points $\bm X_k^n$:
\begin{equation}\label{eqn:dis-bie}
  \frac{1}{2}\psi_k^n + \bm n\cdot\{\grad c_h^n\}_h + \alpha\{c_h^n\}_h = g_k^n,\quad k = 0,1,\cdots,N_\Gamma-1,
\end{equation}
where $\{c_h^n\}_h$ and $\{\grad c_h^n\}_h$ are the approximations of $\{c\}$ and $\{\grad c\}$ at $\bm X_k^n$, respectively.
In particular, using~\eqref{eqn:c-rep}, we compute $\{c_h^n\}_h$ as follows
\begin{equation}
  \{c_h^n\}_h(\bm x) = \frac{1}{2}\Pi_h\paren{c_h^+ +c_h^-}(\bm x, t^n) = \Pi_h\paren{c_h + \paren{\frac12-\chi_T}\Corr_h}(\bm x, t^n), \quad (\bm x,t^n)\in \Sigma_T,
\end{equation}
where $\Pi_h$ denotes the Lagrange interpolation operator.
Since $c_h + (\frac12-\chi_T)\Corr_h$ is a continuous function on $\Sigma_T$, the evaluation of $\{c_h^n\}$ at $\bm X_k^n$ can be done by Lagrange interpolation using values of $c_h$ and $\Corr_h$ at Cartesian grid nodes. For each target point $\bm X_k^n$, we first choose the closest cell center and form the surrounding $3\times3$ block of candidate cell centers. From these candidates we select the six nearest centers for which the Vandermonde matrix associated with the basis $\{1,x,y,x^2,y^2,xy\}$ has full rank. The resulting six values define a local quadratic Lagrange interpolant used to approximate $c_h^n+(\frac12-\chi^n)\Corr^n_h$ at $\bm X_k^n$; see Figure \ref{fig:stencil} for an illustration.

\begin{figure}[htbp]
  \centering
    \includegraphics[width=0.45\textwidth]{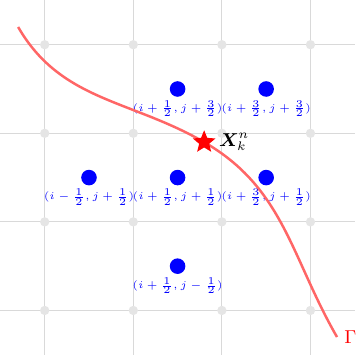}
  \caption{Illustration of the six cell centers used for local quadratic interpolation near an interface point $\bm X_k^n$. The blue dots indicate the selected admissible six-point stencil from the surrounding cell-centered candidates.}
  \label{fig:stencil}
\end{figure}

\subsection{Local Cauchy problem}
To compute the correction function, we need local extensions of $c^+$ and $c^-$ in the narrow band $\Omega_{\Sigma,T}$. We choose the extensions so that they satisfy the same differential equation in the band, with extended source terms. Let $f^\pm$ be smooth extensions of $\at{f}{Q_T}$ and $\at{f}{Q_T^c}$ into $\Omega_{\Sigma,T}$.
Since we take $\at{f}{Q_T^c}=0$, we have $f^-=0$ in the band. For $f^+$, we extend $f$ from the interface by a Taylor expansion. The expansion order is tied to the desired order of the scheme; for the second-order spatial scheme used here, a locally constant extension of $f^+$ is enough.
Hence, $c^\pm$ satisfy
\begin{equation}
\begin{aligned}
\partial_t c^\pm + \grad\cdot\paren{\bm u c^\pm - \grad c^\pm} &= f^\pm, \quad&&\text{ in }\Omega_{\Sigma,T},
\end{aligned}
\end{equation}
together with continuous function values and normal derivatives, $\jump{c^\pm}=0$ and $\jump{\partial_{\bm n}c^\pm}=0$ on $\Sigma_T$.

The exact correction function then satisfies the following equations in the moving narrow band $\Omega_{\Sigma,T}$:
\begin{equation}\label{eqn:cauchy-C}
\begin{aligned}
\partial_t \Corr + \grad\cdot\paren{\bm u\Corr - \grad \Corr} &= F, \quad&&\text{ in }\Omega_{\Sigma,T}, \\
\Corr &= 0, \quad&&\text{ on }\Sigma_T,\\
\partial_{\bm n} \Corr &= \psi, \quad&&\text{ on }\Sigma_T,\\
\Corr &= 0, \quad && \text{ in }\Omega\times \{t=0\},
\end{aligned}
\end{equation}
where $F = f^+ - f^-$. The corresponding boundary Cauchy problem is ill-posed in the sense of Hadamard~\cite{Hadamard1923}. Our goal is therefore not to design a stable solver for the Cauchy problem in the usual sense. Instead, we use these conditions only to locally construct a polynomial that approximates the correction function within the one- to two-grid-cell band where corrections are needed.
At each time step $t^n$, we express the numerical solution $\Corr_h^n$ in the narrow band, around an interface point $\bm p_k=\bm X_k^n$, as
\begin{equation}
    \Corr_h^{n}(\bm x) = \sum_{m=1}^6 a_{k,m}^{n} \phi_m\paren{\frac{\bm x - \bm p_k}{h}},
\end{equation}
where $a_{k,m}^{n}$ are unknown coefficients to be determined and $\phi_m(\bm x)$ are polynomial basis functions satisfying
\begin{equation}
  \phi_m \in \text{span}\left\{ 1, x, y, x^2, y^2, xy \right\}.
\end{equation}
We build the correction function in the narrow band by patching these local polynomial solutions. The correction function is only used to balance the truncation error near irregular stencils, so the patched correction function need not be globally smooth across patch boundaries. We therefore use the simplest patching rule: $\Corr_h^n(\bm x)$ is taken from the local polynomial centered at the closest interface point. Thus the numerical correction is represented by the coefficients $a_{k,m}^n$ at each interface point.

For the local solution around $\bm p_k$, we use a neighborhood $B(\bm p_k)$ with radius of order $h$. We approximate the time derivative by a finite difference and determine the spatial polynomial by collocation. More precisely, we choose three sets of collocation points, $\mathcal{X}_{D,k}^n,\mathcal{X}_{N,k}^n\subset \Gamma(t^n)\cap B(\bm p_k)$ and $\mathcal{X}_{P,k}^n\subset\Omega_{\Gamma(t^n)}\cap B(\bm p_k)$, and impose the PDE and boundary conditions for $\Corr_h^n$ at those points:
\begin{equation}\label{eqn:cauchy-C-h}
\begin{aligned}
  \frac{\Corr_h^n(\bm x)}{\tau} + \grad\cdot\paren{\bm u\Corr_h^n - \grad \Corr_h^n}(\bm x) &= F(\bm x,t^n) + \frac{\Corr_h^{n-1}(\bm x)}{\tau}, \quad&&\text{ for }\bm x\in \mathcal{X}_{P,k}^n,\\
\Corr_h^n(\bm x) &= 0, \quad&&\text{ for }\bm x\in \mathcal{X}_{D,k}^n,\\
\partial_{\bm n} \Corr_h^n(\bm x) &= \psi^n(\bm x), \quad&&\text{ for }\bm x\in \mathcal{X}_{N,k}^n,
\end{aligned}
\end{equation}
together with the initial condition $\Corr_h^0\equiv 0$.
The collocation problem~\eqref{eqn:cauchy-C-h} results in an $N_c$-by-$6$ linear system for the coefficients $a_{k,m}^n$, with $N_c$ being the number of collocation points.
Thus $a_{k,m}^n$ depends linearly on $\psi^n$, $a_{k,m}^{n-1}$, and $F$.
If $N_c>6$, the system is overdetermined and one has to solve it in the least squares sense.
With the collocation points chosen below, the system is square; see~\cite{zhou2024correction}.
Here, we choose $6$ collocation points as follows:
\begin{equation}
\begin{aligned}
  \mathcal{X}_{D,k}^n &= \{ \bm X(s_k+ j\eta\Delta s,t^n), j = -1,0,1\}, \\
  \mathcal{X}_{N,k}^n &= \{ \bm X(s_k+ j\eta\Delta s,t^n), j = -1,1\}, \\
  \mathcal{X}_{P,k}^n &= \{\bm X(s_k,t^n)\},
\end{aligned}
\end{equation}
where $\eta$ is a parameter that controls the distance between collocation points. The PDE collocation set $\mathcal{X}_{P,k}^n$ contains the central interface point; the residual there is evaluated as the trace of the local polynomial residual from the narrow band. We use $\eta=0.5$ in the computations. In our tests, the accuracy is not very sensitive to this choice.
Each local problem is a $6\times6$ linear system and can be solved by a standard dense solver, such as QR factorization. The local systems are independent, so the total cost of building the correction is proportional to the number of interface points.

\subsection{Numerical implementation}

The implementation requires only elementary computational geometry at each time
level. First, the interface markers are looped over to identify Cartesian
cell centers in a narrow band around $\Gamma(t^n)$, assign their side
information, and record the nearest interface marker. The side labels can then
be propagated from this narrow band to the remaining cell centers by a
grid-based breadth-first search, or flood-fill/rasterization step. Meanwhile,
for each interface marker we record nearby Cartesian cell centers used for
interpolation and local collocation. The method does not require cut-cell
volumes, grid--interface intersection points, or numerical quadrature over
partial cells, which keeps the implementation simple.

We now summarize the linear algebra used in the scheme.
Let $\bm c^n$, $\bm \psi^n$, and $\bm a^n$ denote the collections of concentrations at bulk grid points, single-layer densities, and correction function coefficients at the $n$-th time step, respectively.
The collocation equations~\eqref{eqn:cauchy-C-h}, corrected finite difference equation~\eqref{eqn:corrected-scheme}, and discrete boundary condition~\eqref{eqn:dis-bie} give one coupled linear system for $\bm a^n$, $\bm c^n$, and $\bm \psi^n$:
\begin{equation}\label{eqn:unified-system}
\begin{bmatrix}
A^n & 0 & B^n \\
E^n & L^n & 0 \\
M^n & K^n & \frac{1}{2}I
\end{bmatrix}
\begin{bmatrix}
\bm a^n \\
\bm c^n \\
\bm \psi^n
\end{bmatrix}
=
\begin{bmatrix}
\bm b_1^n \\
\bm b_2^n \\
\bm g^n
\end{bmatrix},
\end{equation}
where $A^n, B^n, E^n, L^n, M^n$, and $K^n$ are the matrices associated with the respective discrete operators and $\bm b_1^n,\bm b_2^n$ are vectors that depend linearly on the known values $\bm a^{n-1}, \bm c^{n-1}, \bm \psi^{n-1}$ and $f$.
Let $N_h=\mc{O}(N^2)$ be the number of Cartesian cell-center unknowns and let
$N_\Gamma$ be the number of interface points. With a
six-coefficient local polynomial at each interface point, set
$N_a=6N_\Gamma$. Then
\begin{equation}
\bm a^n\in\R^{N_a},\qquad
\bm c^n\in\R^{N_h},\qquad
\bm\psi^n\in\R^{N_\Gamma}.
\end{equation}
The three block rows of~\eqref{eqn:unified-system} have sizes
$N_a$, $N_h$, and $N_\Gamma$, respectively, and the three block columns
have the same sizes. Equivalently,
\begin{equation}
\begin{aligned}
A^n&\in\R^{N_a\times N_a},&
B^n&\in\R^{N_a\times N_\Gamma},\\
E^n&\in\R^{N_h\times N_a},&
L^n&\in\R^{N_h\times N_h},\\
M^n&\in\R^{N_\Gamma\times N_a},&
K^n&\in\R^{N_\Gamma\times N_h}.
\end{aligned}
\end{equation}
Here, $L^n$ is the matrix associated with the regular-grid operator
$\mc L_{\tau,h}^n$, and $A^n$ is block diagonal, consisting of
$N_\Gamma$ blocks of size $6\times6$. We solve linear systems involving
$L^n$ using geometric multigrid and invert $A^n$ block by block.
Eliminating $\bm c^n$ and $\bm a^n$ gives a system for the interface density $\bm \psi^n$:
\begin{equation}\label{eqn:linear-sys-psi}
  \paren{ \frac{1}{2} I + G^n (A^n)^{-1} B^n } \bm \psi^n = \bm g^n - K^n (L^n)^{-1}\bm b_2^n + G^n (A^n)^{-1}\bm b_1^n,\quad G^n = K^n (L^n)^{-1}E^n - M^n.
\end{equation}
We solve~\eqref{eqn:linear-sys-psi} by the matrix-free GMRES. Each matrix-vector product requires one multigrid solve for $L^n$ and independent inversions of the small blocks of $A^n$. The overall cost is linear in $N_h$ when the GMRES iteration count stays essentially independent of $N_h$; the numerical results below show this behavior.

\section{Numerical analysis}\label{sec:numerical-analysis}
This section proves stability and convergence for a one-dimensional version of
the corrected finite difference scheme. The model problem is an interval with
one prescribed moving endpoint, a constant scalar velocity field, and Robin data at
the moving boundary.
For simplicity, the one-dimensional analysis uses the solution values at grid nodes as unknowns, whereas the two-dimensional implementation uses cell-centered values. This affects only the treatment of the box boundary and does not alter the moving-boundary correction mechanism.

\subsection{Continuous setting}
Let $\mc B=(0,L)$ be the fixed computational domain. We assume $\gamma(t)\in\mc B$ for all $t\in[0,T]$. At time
$t$, the physical interval is $\Omega(t)=(0,\gamma(t))$, and the moving
boundary is $\Gamma(t)=\{\gamma(t)\}$. We write $\dot\gamma$ for the boundary velocity. 
Let $v$ be a constant scalar velocity, and define $\alpha(t) = \dot{\gamma}(t)-v \ge \alpha_0$.
We consider the following problem for the solution $c$:
\begin{equation}\label{eqn:prob-1d}
\begin{aligned}
    \p_t c + v\p_x c &= \p_{xx} c + f,\quad && \text{ in } Q_T, \\
    \p_x c + \alpha c &= g, \quad &&\text{ on }\Sigma_T,\\
    c &= c_0, \quad && \text{ in }\Omega(0)\times\{0\}, \\
    c &= q, \quad && \text{ on }\{0\}\times (0,T].
\end{aligned}
\end{equation}
We assume that $\gamma, c_0, f, g,$ and $q$ are sufficiently regular so that the solution $c \in C^{4,2}(\overline{Q_T})$.
Here, $C^{k,l}(U)$ denotes the space of functions $w$ whose derivatives
$\partial_t^j\partial_x^i w$ are continuous in $U$ for $0\le i\le k$ and
$0\le j\le l$.

\begin{theorem}\label{thm:max_principle}
    Assume $\alpha_0 > 0$.
Then the following statements hold:
\begin{enumerate}
    \item If $f\le 0$ in $Q_T$, $g\le 0$ on $\Sigma_T$, $q\le 0$ on $\{0\}\times(0,T]$, and $c_0\le 0$ in $\Omega(0)$, then $c\le 0$ in $\overline{Q_T}$.
Similarly, if $f\ge 0$, $g\ge 0$, $q \ge 0$, and $c_0\ge 0$, then $c\ge 0$.
\item
If $c$ solves~\eqref{eqn:prob-1d}, then
\begin{equation}
\|c\|_{L^\infty(Q_T)}
\le
\|c_0\|_{L^\infty(\Omega(0))}
+
\sup_{t\in(0,T]}\abs{q(t)}
+
\alpha_0^{-1}\|g\|_{L^\infty(\Sigma_T)}
+
T\|f\|_{L^\infty(Q_T)}.
\end{equation}
\end{enumerate}
\end{theorem}

\begin{proof}
For any $\varepsilon>0$, set $w=c-\varepsilon(t+1)$. Then $w_t+v w_x-w_{xx}=f-\varepsilon<0$ after the usual perturbation argument.
The standard maximum principle for parabolic equations with Robin boundary conditions applies here since $\alpha \ge \alpha_0 > 0$.
Set
\begin{equation}
\begin{aligned}
\mathcal{F}&:=\|f\|_{L^\infty(Q_T)},&
\mathcal{G}&:= \alpha_0^{-1}\|g\|_{L^\infty(\Sigma_T)},\\
\mathcal{C}_0&:=\|c_0\|_{L^\infty(\Omega(0))},&
\mathcal{Q}&:= \sup_{t\in (0,T]}\abs{q(t)}.
\end{aligned}
\end{equation}
Define the comparison function $\overline c(x,t):= \mathcal{C}_0 + \mathcal{G} + \mathcal{Q} + \mathcal{F}t$.
In the interior, $\overline c_t + v \overline c_x - \overline c_{xx} = \mathcal{F} \ge f$.
On the boundary $\Sigma_T$, $\partial_x \overline c + \alpha \overline c = \alpha(\mathcal{C}_0+\mathcal{G}+\mathcal{Q}+\mathcal{F}t) \ge \alpha \mathcal{G} \ge g$, since $\alpha \ge \alpha_0$.
On the boundary $\{0\}\times(0,T]$, $\overline{c}(0,t) \ge \mathcal{Q} \ge q$.
At $t=0$, $\overline c(\cdot,0) = \mathcal{C}_0 + \mathcal{G} + \mathcal{Q} \ge c_0$.
Thus $w:=c-\overline c$ satisfies the hypotheses of the maximum principle, and therefore $c\le \overline c$ in $Q_T$. Applying the same reasoning to $-c$ yields the stated $L^\infty$ estimate.
\end{proof}

If $\alpha_0 < 0$, we choose $\beta$ sufficiently large such that $\beta >\max\{\abs{\alpha_0} + 1, \abs{v}\}$.
With the transform $c=e^{\beta x+(\beta^2-v\beta)t}w$, the problem for $c$
becomes
\begin{equation}\label{eqn:prob-w}
\begin{aligned}
    \p_t w  - \p_{xx} w + (v - 2\beta) \p_x w &= e^{-\beta x - (\beta^2 - v\beta)t} f, && \text{ in } Q_T, \\
    \p_x w + (\alpha + \beta) w &= e^{-\beta x - (\beta^2 - v\beta)t} g, && \text{ on } \Sigma_T, \\
    w &= e^{-\beta x} c_0, && \text{ in } \Omega(0) \times \{t=0\}, \\
    w &= e^{-(\beta^2 - v\beta)t} q, && \text{ on } \{0\} \times (0,T].
\end{aligned}
\end{equation}
The boundary coefficient in~\eqref{eqn:prob-w} satisfies
$\alpha+\beta\ge1>0$. Applying Theorem~\ref{thm:max_principle} to $w$ gives an
$L^\infty$ estimate for $w$, and therefore for $c$.

The corresponding interface problem on the fixed interval is
\begin{equation}
\begin{aligned}
    \p_t c + v\p_x c &= \p_{xx} c + f,\quad && \text{ in }\mc B\times(0,T], \\
    \jump{c} &= 0, \quad && \text{ on }\Sigma_T,\\
    \jump{\p_x c} &= \psi, \quad && \text{ on }\Sigma_T,\\
    \frac{1}{2}\psi + \{ \p_x c +\alpha c\} &= g,\quad && \text{ on }\Sigma_T,\\
    c &= c_0, \quad && \text{ in }\Omega\times\{t=0\}, \\
    c &= q, \quad && \text{ on }\{0\}\times (0,T],\\
    c &= 0, \quad && \text{ on }\{L\}\times (0,T],
\end{aligned}
\end{equation}
where $\psi$ is the interface density.
The artificial boundary condition at $x=L$ closes the extended problem on the
box. It does not impose an additional condition on the physical solution
in $(0,\gamma(t))$: after the interface unknown is eliminated, the interior
rows decouple from the exterior rows, as shown in Proposition~\ref{prop:red-scheme}.
Let $\Omega_{\Gamma,T}\subset \mc B\times(0,T]$ denote a fixed-width space--time neighborhood of
$\Sigma_T$, for example $\{(x,t):0<t\le T,\ |x-\gamma(t)|<r_\Gamma\}$ for a
fixed $r_\Gamma>0$, and let $\Omega_\Gamma(0)=\{x\in\mc B:\ |x-\gamma(0)|<r_\Gamma\}$ be its
initial cross section. The one-dimensional correction function
$\Corr:\Omega_{\Gamma,T}\to \R$ satisfies the following equation.
\begin{equation}
\begin{aligned}
    \p_t \Corr + v\p_x \Corr &= \p_{xx} \Corr + F,\quad &&\text{ in }\Omega_{\Gamma,T}, \\
    \Corr &= 0, \quad &&\text{ on }\Sigma_T,\\
    \p_x \Corr &= \psi, \quad &&\text{ on }\Sigma_T,\\
    \Corr &= 0, \quad &&\text{ in } \Omega_\Gamma(0).
\end{aligned}
\end{equation}
The relation $\p_t \Corr + \dot{\gamma}(t)\p_x\Corr = \frac{d}{dt}\Corr(\gamma(t),t) = 0$ gives
\begin{align}
\p_{xx}\Corr(\gamma(t),t) = -\alpha(t)\psi(t) - F(\gamma(t),t) = -\alpha(t)\psi(t) - f(\gamma(t),t).
\end{align}
Taylor expansion of the correction function at $(\gamma,t)$ gives
\begin{equation}
\begin{aligned}
    \Corr(x,t) &= (x-\gamma(t))\psi(t) - \frac{1}{2}\paren{x-\gamma(t)}^2\paren{\alpha(t)\psi(t) + f(\gamma(t),t)} + \mc O\paren{\paren{x-\gamma(t)}^3}.
\end{aligned}
\end{equation}
\subsection{Discrete setting and notation}
We use a uniform grid on $\mc B$ with spacing $h=L/N$ and nodes $x_j=jh$,
$j=0,1,\dots,N$. We also use time levels $t^n=n\tau$, $n=0,1,\dots,N_t$,
with $\tau=T/N_t$. Let $\mc B_h=\{x_1,x_2,\dots,x_{N-1}\}$ be the set of
interior grid nodes.
Define $\sigma(t) = \frac{\gamma(t)}{h} - \lfloor\frac{\gamma(t)}{h}\rfloor\in [0,1)$.
At a fixed time $t^n$, we define the interface node index $j(n)=\lfloor\frac{\gamma(t^n)}{h}\rfloor$ such that $\gamma^n = \gamma(t^n)\in[x_{j(n)}, x_{j(n)+1})$ and set $\sigma^n = \sigma(t^n)$.
We impose the CFL condition
$\tau < h/\max_{[0,T]}\abs{\dot{\gamma}(t)}$. Then
$j(n)-j(n-1)\in\{-1,0,1\}$, so the interface moves at most one grid cell per
time step.
Throughout the estimates below, $C$ denotes a generic constant that may depend on bounds for the exact solution, $v$, $\alpha$, $\gamma$, the data, $T$, and $L$, but is independent of $h$, $\tau$, and the time index $n$.

Let $c_j^n$ approximate $c(x_j,t^n)$. We collect exact and numerical nodal
values as
\begin{equation}
\bm C^n=(c(x_1,t^n),\dots,c(x_{N-1},t^n))^T,\qquad
\bm c^n=(c_1^n,\dots,c_{N-1}^n)^T .
\end{equation}
Both vectors lie in $\R^{N-1}$.
Thus uppercase bold symbols such as $\bm C^n$ contain exact nodal values, while lowercase bold symbols such as $\bm c^n$ denote numerical unknowns.
We define the restriction operator
$
R^n =  \begin{bmatrix}
        I  & 0
    \end{bmatrix}\in \R^{j(n)\times (N-1)}
$
so that it restricts a grid function on $\mc B_h$ to the interior $\Omega_h^n$.
Let $\delta_{ij}$ be the Kronecker delta and $e_j\in\R^{N-1}$ be the $j$-th standard basis vector.

The finite difference scheme for the interface problem is
\begin{equation}\label{eqn:main_discrete}
\begin{aligned}
&-\paren{\frac{\tau}{h^2}+\frac{v\tau}{2h}}\paren{ c_{j-1}^n - \delta_{j,j(n)+1}\Corr_{j(n)}^n } + \paren{1+\frac{2 \tau}{h^2}}c_j^n \\
& - \paren{\frac{\tau}{h^2}-\frac{v\tau}{2h}}\paren{ c_{j+1}^n + \delta_{j,j(n)}\Corr_{j(n)+1}^n } \\
&= c_j^{n-1} + \delta_{j(n-1), j(n)-1}  \delta_{j,j(n)}\Corr_{j(n)}^{n-1}  - \delta_{j(n-1),j(n)+1}\delta_{j,j(n)+1}\Corr_{j(n)+1}^{n-1} + \tau f(x_j, t^n) .
\end{aligned}
\end{equation}
Here $\Corr_j^n$ denotes the approximate correction function, with
\begin{align}
    \Corr_{j(n)}^n &=  a_{n,1} h \psi^n + b_{n,1}h^2 f(\gamma^n,t^n), \\
    \Corr_{j(n)+1}^n &= a_{n,2} h \psi^n  + b_{n,2}h^2 f(\gamma^n,t^n).
\end{align}
At the grid node $x_1$, the Dirichlet condition $c=q$ gives
\begin{equation}
\begin{aligned}
\paren{1+\frac{2 \tau}{h^2}}c_1^n  - \paren{\frac{\tau}{h^2}-\frac{v\tau}{2h}} c_2^n = c_1^{n-1} +  \tau f(x_1,t^n) + \paren{\frac{\tau}{h^2}+\frac{v\tau}{2h}} q^n,
\end{aligned}
\end{equation}
where the coefficients are
\begin{align}
    &a_{n,1} = -\sigma^n\paren{1+\frac{1}{2}\alpha^n\sigma^n h},
    &&b_{n,1} = -\frac{1}{2}\paren{\sigma^n}^2,\\
    &a_{n,2} = \paren{1-\sigma^n}\paren{ 1-\frac{1}{2}\alpha^n\paren{1-\sigma^n}h},
    &&b_{n,2} = -\frac{1}{2}\paren{1-\sigma^n}^2.
\end{align}
The Robin condition $\p_x c+\alpha c=g$ is discretized as
\begin{equation}\label{eqn:bc-discrete}
\begin{aligned}
\frac{ \sigma^n-\frac{1}{2} }{ h} c_{j(n)-1}^n + \frac{ -2\sigma^n+\paren{1-\sigma^n}\alpha^nh}{h} c_{j(n)}^n  + \frac{\sigma^n+\frac{1}{2}+\sigma^n\alpha^nh} {h}\paren{c_{j(n)+1}^n + \Corr_{j(n)+1}^n} = g^n.
\end{aligned}
\end{equation}
The three-point stencil in~\eqref{eqn:bc-discrete} differs slightly from the proximity-based interpolation rule in Section~\ref{sec:method}. We fix the stencil as $x_{j(n)-1}$, $x_{j(n)}$, and $x_{j(n)+1}$, so it always contains two interior nodes and one exterior node, regardless of the position of $\gamma^n$ in $[x_{j(n)},x_{j(n)+1})$. This allows us to decouple the interior unknowns from the exterior ones; see Proposition~\ref{prop:red-scheme}.

We may write the discrete system in vector form using the following matrices and vectors:
\begin{align}
    A &=
\begin{bmatrix}
2 & -1   &        &  \\
-1 & 2   &        &  \\
   &    & \ddots & -1 \\
   &       & -1 & 2
\end{bmatrix},\quad
B =
\begin{bmatrix}
0 & 1     &        &  \\
-1 & 0    &        &  \\
   &     & \ddots & 1 \\
   &       & -1 & 0
\end{bmatrix},\label{eqn:AB-mat}\\
    E^n &= a_{n,2}\paren{1-\frac{vh}{2}} e_{j(n)} -a_{n,1} \paren{1+\frac{vh}{2}} e_{j(n)+1},\\
    D^n &= \paren{\sigma^n-\frac{1}{2}}e_{j(n)-1} + \paren{-2\sigma^n+\paren{1-\sigma^n}\alpha^nh} e_{j(n)} + \paren{\sigma^n+\frac{1}{2}+\sigma^n\alpha^nh}e_{j(n)+1},\\
    Q^n &= \paren{\sigma^n+\frac{1}{2}+\sigma^n\alpha^nh} a_{n,2},\\
    Z^{n-1} &= \delta_{j(n-1), j(n)-1}a_{n-1,2}e_{j(n)} - \delta_{j(n-1),j(n)+1}a_{n-1,1}e_{j(n)+1}.
\end{align}

For $n=1,2,\cdots,N_t$, the scheme is
\begin{subequations}\label{eqn:scheme}
\begin{align}
    \paren{I + \frac{\tau}{h^2}A + \frac{v\tau}{2h}B}\bm c^n - \frac{\tau}{h}E^n \psi^n &= \bm c^{n-1} + h  Z^{n-1}\psi^{n-1} + \tau \wt{\bm f}^n + \tau\rho  q^n e_1,\\
    \frac{1}{h} (D^n)^T\bm c^n +  Q^n\psi^n &= \wt g^n,
\end{align}
\end{subequations}
where $\rho = \frac{1}{h^2}+\frac{v}{2h}$ and
\begin{equation}
\begin{aligned}
\wt f_j^n
    &= f_j^n + \paren{ - \paren{1+\frac{vh}{2}}\delta_{j, j(n)+1}  b_{n,1}  + \paren{1-\frac{vh}{2}} \delta_{j, j(n)} b_{n,2}} f(\gamma^n,t^n) \\
    &+ \paren{\delta_{j,j(n)} \delta_{j(n-1), j(n)-1}b_{n-1,2}   - \delta_{j,j(n)+1}\delta_{j(n-1),j(n)+1}b_{n-1,1} }\frac{h^2}{\tau}f(\gamma^{n-1},t^{n-1}).
\end{aligned}
\end{equation}
The modified boundary data are
\begin{equation}
    \wt g^n = g^n - D^n_{j(n)+1} b_{n,2}hf(\gamma^n,t^n).
\end{equation}

\subsection{Linear algebra}
We establish a few linear algebraic properties of the numerical scheme. 
We first prove solvability.
Define
\begin{align}
    K^n = I + \frac{\tau}{h^2}\paren{A + \frac{vh}{2} B + \frac{1}{Q^n} E^n (D^{n})^T}, \quad M^{n-1} &=  I - \frac{1}{Q^{n-1}}Z^{n-1}(D^{n-1})^T.
\end{align}
For sufficiently small $h$, $Q^n>0$, so $K^n$ and $M^{n-1}$ are well defined. Since $Q^n$ contains the factor $1-\sigma^n$, it can become small as $\sigma^n$ approaches $1$. In the interior row, however, this factor cancels between $E^n$ and $Q^n$; a factor $(1-\sigma^n)^{-1}$ can remain only in the exterior row associated with $x_{j(n)+1}$. We will further show that the interior solution is decoupled from the exterior ones, so a small value of $Q^n$ does not enter the stability estimate for the interior solution $R^n\bm c^n$.

\begin{proposition}\label{prop:solvability}
For sufficiently small $h$, the scheme~\eqref{eqn:scheme} is solvable.
\end{proposition}
\begin{proof}
Since $Q^n > 0$ uniformly for sufficiently small $h$, we can apply a row operation to the coefficient matrix of~\eqref{eqn:scheme} to obtain an equivalent block-triangular matrix
\begin{equation}\label{eqn:eqv-block}
\begin{bmatrix}
K^n & 0\\
\frac{1}{h}(D^n)^T & Q^n
\end{bmatrix}.
\end{equation}

Let $\wt A = A + \frac{vh}{2} B + \frac{1}{Q^n} E^n (D^{n})^T$.
For sufficiently small $h$, the matrix $I + \frac{\tau
}{h^2}A+\frac{v\tau}{2h}B$ is an M-matrix and is invertible. Note that the matrix $\frac{1}{Q^n}E^n(D^n)^T$ is nonzero only for the block $j(n):j(n)+1, j(n)-1:j(n)+1$.
A direct calculation gives
\begin{equation}
\begin{aligned}
    &\wt{A}_{j(n), j(n)+1}  = -1 + \frac{vh}{2} + \frac{\paren{1-\frac{vh}{2}} \paren{\sigma^n+\frac{1}{2}+\sigma^n\alpha^nh} }{\paren{\sigma^n+\frac{1}{2}+\sigma^n\alpha^nh} } = 0.
\end{aligned}
\end{equation}
Hence, the block $K^n_{j(n):j(n), j(n)+1:N-1}$ is zero.
The remaining entries needed below are
\begin{equation}
\begin{aligned}
    &\wt{A}_{j(n), j(n)-1}  = -1 - \frac{vh}{2} + \frac{\paren{1-\frac{vh}{2}} \paren{\sigma^n-\frac{1}{2}} }{\paren{\sigma^n+\frac{1}{2}+\sigma^n\alpha^nh} } = -1 + \frac{2\sigma^n-1}{2\sigma^n+1} + \mc O(h), \\
    &\wt{A}_{j(n), j(n)} = 2 + \frac{\paren{1-\frac{vh}{2}} \paren{-2\sigma^n+\paren{1-\sigma^n}\alpha^n\sigma^nh} }{\paren{\sigma^n+\frac{1}{2}+\sigma^n\alpha^nh} } = 2 - \frac{4\sigma^n}{2\sigma^n+1} + \mc O(h),  \\
    &\wt{A}_{j(n)+1, j(n)+1} = 2 - \frac{a_{n,1}\paren{1+\frac{vh}{2}} }{a_{n,2} } = 2 + \frac{\sigma^n}{1-\sigma^n} + \mc O(h).
\end{aligned}
\end{equation}
For sufficiently small $h$, the displayed entries have the signs required by the M-matrix argument. Then, both diagonal blocks of $K^n$ are M-matrices and invertible. Hence, $K^n$ is invertible. Combining the fact that $Q^n$ is nonzero, we conclude that the matrix~\eqref{eqn:eqv-block} is invertible, and so is the original system in~\eqref{eqn:scheme}.
\end{proof}

As a consequence, the discrete boundary integral equation obtained from the
Schur complement of~\eqref{eqn:scheme} is solvable.
\begin{corollary}\label{cor:schur-inv}
For sufficiently small $h$, the matrix
\begin{equation}
    Q^n + \frac{\tau}{h^2} (D^n)^T\paren{I + \frac{\tau
}{h^2}A+\frac{v\tau}{2h}B}^{-1} E^n
\end{equation}
is invertible.
\end{corollary}
\begin{proof}
Since the top-left block $I+\frac{\tau}{h^2}A + \frac{v\tau}{2h}B$ is invertible for sufficiently small $h$, by Proposition~\ref{prop:solvability}, the Schur complement is also invertible.
\end{proof}

We next show that the numerical scheme can be written purely in terms of $R^n\bm{c}^n$, the values of $\bm{c}^n$ within the interior region $\Omega^n_h$.
\begin{proposition}\label{prop:red-scheme}
For sufficiently small $h$, the interior solution $R^n\bm c^n$ satisfies:
\begin{equation}\label{eqn:Rn-un-eqn}
\begin{aligned}
    R^n K^n (R^n)^T R^n \bm  c^n &=  R^n M^{n-1} (R^{n-1})^T R^{n-1} \bm  c^{n-1} \\
    &+ R^n\paren{ \tau \paren{\wt{\bm f}^n + \rho  q^n e_1 + \frac{1}{h Q^n} E^n \wt g^n} + \frac{h}{Q^{n-1}}  Z^{n-1} \wt g^{n-1} },
\end{aligned}
\end{equation}
for $n=1,2,\dots, N_{t}$.
\end{proposition}
\begin{proof}
For sufficiently small $h$, $Q^n > 0$ uniformly, so we have
\begin{equation}
    \psi^{n-1} = -\frac{1}{h Q^{n-1}} (D^{n-1})^T\bm c^{n-1} + \frac{\wt g^{n-1}}{Q^{n-1}}.
\end{equation}
Eliminating $\psi^{n-1}$ and $\psi^n$ from~\eqref{eqn:scheme} gives the full-grid equation
\begin{align}\label{eqn:un-eqn}
    K^n \bm  c^n  =  M^{n-1}  \bm  c^{n-1} + \tau\paren{\wt{\bm f}^n + \rho  q^n e_1 + \frac{1}{hQ^n} E^n \wt g^n} + \frac{h}{Q^{n-1}}  Z^{n-1} \wt g^{n-1}.
\end{align}
The matrix $K^n$ is tridiagonal and $K^n_{j(n),j(n)+1}=0$, so the top-right block of $K^n$ is zero. Hence the interior unknown $R^n\bm c^n$ decouples from the exterior unknowns. Restricting~\eqref{eqn:un-eqn} to the interior gives~\eqref{eqn:Rn-un-eqn}.
\end{proof}
\begin{remark}
The reason why interior and exterior variables decouple can also be seen as follows. Coupling between interior and exterior variables can only take place when $j=j(n)$ in Eq. \eqref{eqn:main_discrete} through the following two terms. The first is the term on the second line of \eqref{eqn:main_discrete}. The second is the second term on the last line, which is non-zero when $j(n)=j(n-1)+1$ (when the boundary point $\gamma(t)$ crosses one grid point in the positive direction). Note that both terms are constant multiples of $c_{j(m)+1}^m+\Corr_{j(m)+1}^m, m=n,n-1$. Thus, we can replace these terms with linear combinations of $c_{j(m)-1}^m$ and $c_{j(m)}^m$ using the discretization of the boundary condition \eqref{eqn:bc-discrete}.
\end{remark}

\subsection{Stability}

We next prove the $\ell^\infty$ estimate using a discrete maximum principle. We first employ the exponential weighting in \eqref{eqn:prob-w} so that the maximum principle is applicable. Fix $\beta$ as in the continuous exponential transform and set $\lambda=\beta^2-v\beta$ and
$W=\operatorname{diag}(e^{\beta ih})_{i=1}^{N-1}$. We introduce the transform
\begin{equation}\label{eqn:weighted-transform}
    \bm w^n = (1-\lambda\tau)^n W^{-1}\bm c^n,
    \qquad
    \bm c^n = (1-\lambda\tau)^{-n}W\bm w^n .
\end{equation}
The transformed form of~\eqref{eqn:Rn-un-eqn} is
\begin{equation}\label{eqn:w-eqn}
\begin{aligned}
    &\paren{1-\lambda \tau}^{-n} R^n K_w^n (R^n)^T R^n \bm  w^n =  \paren{1-\lambda \tau}^{-n+1} R^n M_w^{n-1} (R^{n-1})^T R^{n-1} \bm  w^{n-1} \\
    &+ R^nW^{-1}\paren{ \tau \paren{\wt{\bm f}^n + \rho  q^n e_1 + \frac{1}{h Q^n} E^n \wt g^n} + \frac{h}{Q^{n-1}}  Z^{n-1} \wt g^{n-1} },
\end{aligned}
\end{equation}
where
\begin{equation}
    K_w^n = W^{-1} K^n W, \quad M_w^{n-1} = W^{-1} M^{n-1} W.
\end{equation}
For later use, define the interior-grid matrices
\begin{equation}\label{eqn:transfer-matrices}
    K_I^n = R^nK_w^n(R^n)^T\in\R^{j(n)\times j(n)},\quad
    T^n = R^nM_w^{n-1}(R^{n-1})^T\in\R^{j(n)\times j(n-1)}.
\end{equation}
We will also use the following matrix
\begin{equation}
    P^n=R^{n+1}M_w^n(R^n)^T(K_I^n)^{-1}
    \in\R^{j(n+1)\times j(n)}.
\end{equation}

\begin{lemma}\label{lem:KM-mat-o2}
Assume $h<2/\abs{v}$ when $v\ne0$ and $1-\lambda\tau>0$.
For sufficiently small $h$ and $\tau$, the following
statements hold for $n = 1,2,\cdots, N_t-1$:
\begin{enumerate}
    \item The matrix $K_I^n$ is an M-matrix and satisfies
\begin{equation}
    \norm{(K_I^n)^{-1}}_\infty \le 1+C\tau,
\end{equation}
    \item The bound $\norm{T^n}_\infty \le \frac{5}{3}$ holds. If $\tau / h^2 \ge 1/2$, the matrix $P^n$ has nonnegative entries and satisfies
\begin{equation}
    \norm{P^n}_\infty \le 1+C\tau.
\end{equation}
\end{enumerate}
\end{lemma}

\begin{proof}
To prove (1), we consider the entries of the matrix $K_w^n$ since $K_I^n$ is the top-left $j(n)\times j(n)$ block of $K_w^n$.
The matrix $\frac{1}{Q^n}E^n(D^n)^T$ is nonzero only for the block $j(n):j(n)+1, j(n)-1:j(n)+1$.
So, for $i \neq j(n)$, assuming $h < 2 / \abs{v}$, we have
\begin{align}
    (K_w^n)_{i,i-1} &= -\paren{\frac{\tau}{h^2} + \frac{v \tau}{2 h} } e^{-\beta h} < 0, \\
    (K_w^n)_{i,i+1} &= -\paren{\frac{\tau}{h^2} - \frac{v \tau}{2 h} } e^{\beta h} < 0, \\
    (K_w^n)_{ii} &= 1 + \frac{2\tau}{h^2} > 0.
\end{align}
For sufficiently small $h$,
\begin{equation}
\begin{aligned}
    (K_w^n)_{i,i-1} + (K_w^n)_{i,i} + (K_w^n)_{i,i+1}
    &= 1 + \frac{2\tau}{h^2} \paren{1 - \cosh(\beta h)} + \frac{v \tau}{h} \sinh(\beta h)  \ge 1 - C \tau.
\end{aligned}
\end{equation}
We consider the $j(n)$-th row of $K_I^n$,
\begin{equation}
\begin{aligned}
    (K_w^n)_{j(n),j(n)-1} &= e^{-\beta h}\frac{\tau}{h^2}\paren{ -1 - \frac{vh}{2} + \frac{\paren{1-\frac{vh}{2}} \paren{\sigma^n-\frac{1}{2}} }{\paren{\sigma^n+\frac{1}{2}+\sigma^n\alpha^nh} } },
\end{aligned}
\end{equation}
\begin{equation}
\begin{aligned}
    (K_w^n)_{j(n),j(n)} &=  1 + \frac{\tau}{h^2}\paren{2 + \frac{\paren{1-\frac{vh}{2}} \paren{-2\sigma^n+\paren{1-\sigma^n}\alpha^n\sigma^nh} }{\paren{\sigma^n+\frac{1}{2}+\sigma^n\alpha^nh} }}.
\end{aligned}
\end{equation}
Then, we have
\begin{equation}
\begin{aligned}
    (K_w^n)_{j(n),j(n)-1} + (K_w^n)_{j(n),j(n)}  \ge 1 + \frac{2\tau}{h} \paren{  \frac{\beta + \alpha^n\sigma^n(2-\sigma^n)}{2\sigma^n+1} }- C\tau\ge 1 - C\tau,
\end{aligned}
\end{equation}
where we have used
\begin{equation}
    \beta + \alpha^n\sigma^n(2-\sigma^n) \ge \beta + \alpha_0\sigma^n(2-\sigma^n) \ge 1 +\abs{\alpha_0} - \abs{\alpha_0}  \ge 1.
\end{equation}
Thus, for sufficiently small $h$, we have
\begin{equation}\label{eqn:kw-rowsum}
    K_I^n\bm 1 \ge \paren{1 - C\tau}\bm 1,
\end{equation}
where $C>0$ depends on $v$ and $\alpha$.
Since the diagonal elements are positive and the off-diagonal elements are non-positive for sufficiently small $h$ and $\tau$, \eqref{eqn:kw-rowsum} implies that $K_I^n$ is an M-matrix and satisfies
\begin{equation}
    \norm{(K_I^n)^{-1}}_\infty \le 1 + C\tau.
\end{equation}

We now prove (2).
If $j(n)-1 \neq j(n-1)$, the entries of the matrix $T^n=R^n M_w^{n-1} (R^{n-1})^T$ are nonzero only along the diagonal and are all equal to $1$, yielding $\norm{T^n}_\infty=1$.
In the case $j(n)-1 = j(n-1)$, we have
\begin{align}
    \paren{M_w^{n-1}}_{j(n), j(n)-2} &= e^{-2\beta h}\frac{\frac{1}{2}-\sigma^{n-1}}{\sigma^{n-1}+\frac{1}{2}+\sigma^{n-1}\alpha^{n-1}h},\\
    \paren{M_w^{n-1}}_{j(n), j(n)-1} &= e^{-\beta h}\frac{2\sigma^{n-1} - \paren{1-\sigma^{n-1}} \alpha^{n-1} h }{ \frac{1}{2} + \sigma^{n-1} + \sigma^{n-1} \alpha^{n-1} h }.
\end{align}
For $\sigma^{n-1} \in[0,1/2]$ and sufficiently small $h$, if $\alpha^{n-1}\ge 0$, then $\paren{M_w^{n-1}}_{j(n), j(n)-2} \ge 0$, and we have
\begin{equation}\label{eqn:Mwrowjn}
\begin{aligned}
    &\abs{\paren{M_w^{n-1}}_{j(n), j(n)-2}} + \abs{\paren{M_w^{n-1}}_{j(n), j(n)-1}} \\
    &= \frac{ e^{-2\beta h}\paren{\frac{1}{2}-\sigma^{n-1}} + e^{-\beta h}\paren{ 2\sigma^{n-1} - \paren{1-\sigma^{n-1}} \alpha^{n-1} h }}{ \frac{1}{2} + \sigma^{n-1} + \sigma^{n-1} \alpha^{n-1} h } \\
    &\le 1 - \frac{\paren{\beta+\alpha^n}h + \mc O(h^2) }{ \frac{1}{2} + \sigma^{n-1} + \sigma^{n-1} \alpha^{n-1} h } \le 1.
\end{aligned}
\end{equation}
If $\alpha^{n-1}<0$, then it is possible that $\paren{M_w^{n-1}}_{j(n), j(n)-2} < 0$, in which case we have
\begin{equation}
\begin{aligned}
    &\abs{\paren{M_w^{n-1}}_{j(n), j(n)-2}} + \abs{\paren{M_w^{n-1}}_{j(n), j(n)-1}} \\
    &= \frac{ e^{-2\beta h}\paren{\frac{1}{2}-\sigma^{n-1}} - e^{-\beta h}\paren{ 2\sigma^{n-1} - \paren{1-\sigma^{n-1}} \alpha^{n-1} h }}{ \frac{1}{2} + \sigma^{n-1} + \sigma^{n-1} \alpha^{n-1} h } \\
     &\le \frac{ e^{-2\beta h}\paren{\frac{1}{2}-\sigma^{n-1}} + e^{-\beta h} \alpha^{n-1} h }{ \frac{1}{2} + \sigma^{n-1} + \sigma^{n-1} \alpha^{n-1} h } \\
     &= \frac{ \frac{1}{2}-\sigma^{n-1} + \paren{\alpha^{n-1} - (1-2\sigma^{n-1})\beta}h + \mc O(h^2) }{ \frac{1}{2} + \sigma^{n-1} + \sigma^{n-1} \alpha^{n-1} h }  \le 1.
\end{aligned}
\end{equation}
This implies that for $\sigma^{n-1} \in[0,1/2]$, we also have $\norm{T^n}_\infty \le 1$  and
\begin{equation}
    \norm{P^n}_\infty \le \norm{ R^{n+1} M_w^{n}  (R^{n})^T}_\infty \norm{(K_I^n)^{-1} }_\infty \le 1 + C\tau.
\end{equation}
If $\sigma^{n-1} \in (1/2,1)$, the only row whose absolute row sum can exceed
one is again the moving-interface row. For sufficiently small $h$, the second
entry in that row is positive, and therefore
\begin{equation}
\begin{aligned}
    \sum_k\abs{(T^n)_{j(n),k}}
    &\le
    \frac{ e^{-2\beta h}\paren{\sigma^{n-1}-\frac12}
    + e^{-\beta h}\abs{2\sigma^{n-1}-(1-\sigma^{n-1})\alpha^{n-1}h}}
    {\frac12+\sigma^{n-1}+\sigma^{n-1}\alpha^{n-1}h}  \\
    &=
    \frac{3\sigma^{n-1}-\frac12}{\sigma^{n-1}+\frac12}+\mc O(h)
    \le \frac53 .
\end{aligned}
\end{equation}
The last inequality is uniform for sufficiently small $h$. Hence $\norm{T^n}_\infty\le5/3$.

Now, we consider the matrix $P^n=R^{n+1} M_w^{n} (R^{n})^T(K_I^n)^{-1}$.
If $\sigma^{n}\in[0,1/2]$, the non-negativity and norm estimate follows immediately from the preceding row-sum bound and the M-matrix inverse bound.

For $\sigma^n \in (1/2,1)$, let
$\bm z=R^{n+1}M_w^n(R^n)^T\bm w$ and
$\bm w=(R^nK_w^n(R^n)^T)^{-1}\bm y$, where
\begin{equation}
    R^n K_w^n (R^n)^T \bm w = \bm y.
\end{equation}
Since $(K_I^n)^{-1}$ is non-negative and $R^{n+1} M_w^{n} (R^{n})^T$ is nonzero only along the diagonal and equals $1$ for rows from $1$ to $j(n)$, we only need to check the $(j(n)+1)$-th row of $P^n$. Here, $c_{j(n)+1}$ is a ghost-node value on the box $\mc B$, as in the interface formulation.
Since $c_j = w_j$ for $j \le j(n)$, the ghost value $c_{j(n)+1}$ satisfies
\begin{align}
    &e^{-\beta h}\paren{\sigma^n-\frac{1}{2}} w_{j(n)-1} + \paren{-2\sigma^n + \paren{1-\sigma^n}\alpha^n h}w_{j(n)} +e^{\beta h} \paren{\sigma^n + \frac{1}{2} + \sigma^n \alpha^n h} c_{j(n)+1} = 0,\\
    &- e^{-\beta h}\paren{\frac{\tau}{h^2}+\frac{v\tau}{2h}}w_{j(n)-1} + \paren{1+2\frac{\tau}{h^2}} w_{j(n)}  - e^{\beta h} \paren{\frac{\tau}{h^2} - \frac{v\tau}{2h}} c_{j(n)+1} = y_{j(n)}.
\end{align}
By eliminating $w_{j(n)-1}$, we obtain
\begin{align}
c_{j(n)+1} &= e^{-\beta h} \frac{\eta_1 y_{j(n)} + \eta_2 w_{j(n)}}{\eta_3}, \quad \eta_1 = \sigma^n-\frac{1}{2}, \\
    \eta_2 &= \frac{1}{2}-\sigma^n + \frac{\tau}{h^2} \paren{ 1-(1-\sigma^n)\alpha^nh+\frac{vh}{2}\Bigl(2\sigma^n-(1-\sigma^n)\alpha^nh\Bigr)},  \\
    \eta_3 &= \frac{\tau}{h^2}\paren{ 1+\sigma^n\alpha^nh + \frac{vh}{2} \sigma^n \paren{2+\alpha^nh}}.
\end{align}
For $\frac{\tau}{h^2}>\frac{1}{2}$ and sufficiently small $h$, since $\sigma^n\in (\frac{1}{2},1)$, we have
\begin{equation}
    \eta_1, \eta_3 > 0,\quad \eta_2 \ge \frac{1}{2}-\sigma^n + \frac{1}{2}>0.
\end{equation}
This implies the $(j(n)+1)$-th row of $P^n$ is also non-negative and
\begin{equation}
    \abs{c_{j(n)+1}} \le e^{-\beta h}\frac{\eta_1+\eta_2}{\eta_3}\max\{\abs{y_{j(n)}}, \abs{w_{j(n)}}\}.
\end{equation}
Since $e^{-\beta h} = 1 - \beta h + \mc O(h^2)$, by a direct calculation, we have
\begin{equation}
    e^{-\beta h} \frac{\eta_1+\eta_2}{\eta_3} = 1 - (\beta+\alpha^n)h + \mc O(h^2).
\end{equation}
Using the condition $\beta + \alpha^n \ge 1$, we have $e^{-\beta h} \frac{\eta_1+\eta_2}{\eta_3} \le 1 - h + \mc O(h^2) \le 1$ for sufficiently small $h$. This implies
\begin{equation}
    \abs{c_{j(n)+1}} \le \paren{1+C\tau}\norm{\bm y}_\infty,
\end{equation}
and so
\begin{equation}
    \norm{P^n}_\infty \le 1+C\tau.
\end{equation}
\end{proof}

\begin{theorem}\label{thm:mp-o2}
\textcolor{blue}{Suppose $R^0\bm c^0=0$, $q^n \ge 0$, $R^n\wt{\bm f}^n\ge 0$, and $\wt g^n\ge 0$.}
Then $R^n \bm c^n \ge 0$ for $n=0,1,\cdots, N_t$ and
\begin{equation}
    \textcolor{blue}{\norm{R^n\bm c^n}_\infty
    \le e^{CT+\beta L}\paren{
    T\sup_{1\le n\le N_t}\norm{R^n\wt{\bm f}^n}_\infty
    +\sup_{0\le n\le N_t}\abs{\wt g^n}
    +\sup_{1\le n\le N_t}\abs{q^n}}.}
\end{equation}
\end{theorem}
\begin{proof}
Let $\bm w^n=(1-\lambda\tau)^nW^{-1}\bm c^n$, equivalently
$\bm c^n = \paren{1-\lambda\tau}^{-n}W\bm w^n$, as in
\eqref{eqn:weighted-transform}. Then $\bm w^n$ satisfies
\eqref{eqn:w-eqn}.
We will prove that $R^{n+1}M_w^n(R^n)^TR^n \bm w^n \ge 0$ for $n=1,2,\dots$.
We prove the statement by induction.
\textcolor{blue}{Since $R^0\bm c^0=0$, we have $R^0\bm w^0=0$ and hence
$T^1R^0\bm w^0=0$. This gives the first induction step.}

Suppose $T^n R^{n-1}\bm w^{n-1} \ge 0$. Using \eqref{eqn:w-eqn}, we have
\begin{equation}
\begin{aligned}
     R^n \bm w^n &=  (1-\lambda \tau)(K_I^n)^{-1}T^n R^{n-1} \bm w^{n-1} \\
     &+ (1-\lambda \tau)^{n} (K_I^n)^{-1}\tau  R^nW^{-1} \paren{ \wt{\bm f}^n + \rho q^n e_1+ \frac{1}{h Q^n}  E^n \wt g^n} \\
    &+ (1-\lambda \tau)^{n}(K_I^n)^{-1} \frac{h}{Q^{n-1}} R^n W^{-1}Z^{n-1} \wt g^{n-1}.
\end{aligned}
\end{equation}
and
\begin{equation}
\begin{aligned}
     R^{n+1}M_w^n(R^n)^TR^n \bm w^n &=  (1-\lambda \tau)P^nT^n R^{n-1} \bm w^{n-1} \\
     &+(1-\lambda \tau)^{n} P^n\tau  R^nW^{-1}  \paren{ \wt{\bm f}^n + \rho q^n e_1+ \frac{1}{h Q^n}  E^n \wt g^n} \\
    &+(1-\lambda \tau)^{n} P^n \frac{h}{Q^{n-1}} R^n W^{-1}Z^{n-1} \wt g^{n-1}.
\end{aligned}
\end{equation}
By Lemma~\ref{lem:KM-mat-o2}, $P^n$ has nonnegative entries. Also, for sufficiently small $h$, the vectors $\frac{1}{Q^n}R^n E^n$ and $\frac{1}{Q^{n-1}}R^nZ^{n-1}$ are nonnegative, so $R^n\bm w^n \ge 0$ and $R^{n+1}M_w^n(R^n)^TR^n \bm w^n \ge 0$.
Let
\begin{align}
    \textcolor{blue}{\mathcal{F} = \sup_{1\le n\le N_t}
    \norm{R^n\wt{\bm f}^n}_\infty,\quad
    \mathcal{G}=\sup_{0\le n\le N_t}\abs{\wt g^n},\quad
    \mathcal{Q} = \sup_{1\le n\le N_t}\abs{q^n}.}
\end{align}
We now compare $R^n\bm w^n$ with a scalar supersolution.

We introduce $\bm s^n = s^n\bm 1$ with
\begin{equation}
    \textcolor{blue}{s^n=n\tau\mathcal F+\mathcal G+\mathcal Q},
    \quad \lambda = \beta^2-v\beta.
\end{equation}
We have
\begin{equation}
\begin{aligned}
    &(1-\lambda \tau)^{-1} R^n K_w^n (R^n)^T R^n \bm s^n -  R^n M_w^{n-1} (R^{n-1})^T R^{n-1} \bm s^{n-1} \\
    &= \tau R^n \paren{ \mathcal{F}\bm 1 + \rho s^n e_1
    + \frac{1}{hQ^n}E^n  s^n (\beta+\alpha^n+\mc O(h))} \\
    &\quad + \frac{h}{Q^{n-1}}R^nW^{-1}Z^{n-1}
    s^n \paren{\beta+\alpha^{n-1}+\mc O(h)}.
\end{aligned}
\end{equation}
Let $R^n\bm e^n = R^n\bm s^n - R^n\bm w^n$. Then $R^0\bm e^0 = s^0\bm 1 \ge 0$ and
\begin{equation}
\begin{aligned}
    (1-\lambda\tau)^{n-1}W^{-1}\wt{\bm f}^n \le \mathcal{F},\quad
    (1-\lambda\tau)^{n-1}e^{-\beta h}q^n\le s^n,\quad
    (1-\lambda\tau)^{n-1}e^{-\beta x_{j(n)}}\wt g^n \le s^n.
\end{aligned}
\end{equation}
\textcolor{blue}{Thus,
$R^n\bm w^n\le T\mathcal F+\mathcal G+\mathcal Q$ component-wise.
Applying the same argument to $-R^n\bm w^n$ gives
$R^n\bm w^n\ge-T\mathcal F-\mathcal G-\mathcal Q$, and hence
$\norm{R^n\bm w^n}_\infty\le T\mathcal F+\mathcal G+\mathcal Q$.}
Using $(1-\lambda\tau)^{-n} \le  \paren{1+C\tau}^n \le e^{CT}$ proves the estimate.
\end{proof}

\begin{remark}
The lower bound $\tau/h^2\ge 1/2$ in Lemma~\ref{lem:KM-mat-o2} is a technical
assumption of the comparison proof, not an algorithmic restriction. Together
with $\tau<h/\max_{[0,T]}|\dot\gamma|$, it gives the proof range
\begin{equation}
\frac12 h^2\le \tau < \frac{h}{\max_{[0,T]}|\dot\gamma|}.
\end{equation}
This range is nonempty for bounded interface speed and sufficiently small $h$.
The lower bound is used only to prove the non-negativity of one row in the
product matrix. 
Similar mesh-dependent restrictions arise in Cartesian grid methods for
surface PDEs~\cite{Beale2020_SINUM}.
\end{remark}

\subsection{Error estimate}

Let $\bm \varepsilon_c^n=\bm C^n-\bm c^n$ and $\varepsilon_\psi^n=\psi(t^n)-\psi^n$ be the numerical errors. The local truncation error has the following form.
\begin{lemma}\label{lem:LTE}
For $n=1,2,\cdots,N_t$,
\begin{align}
    \paren{I + \frac{\tau}{h^2}A + \frac{v\tau}{2h}B} \bm \varepsilon_c^n - \frac{\tau}{h} E^n \varepsilon_\psi^n &= \bm \varepsilon_c^{n-1} + h  Z^{n-1}\varepsilon_\psi^{n-1} + \tau \sum_{k=1}^3 \bm r_k^n ,\\
    \frac{1}{h} (D^n)^T \bm \varepsilon_c^n +  Q^n\varepsilon_\psi^n &=  r_4^n.
\end{align}
where $\bm r_k^n$ for $k=1,2,3$ and $r_4^n$  satisfy
\begin{align}
\bm r_k^n = r_{k,j(n)}^n e_{j(n)} + r_{k,j(n)+1}^n e_{j(n)+1} , \quad k = 2,3,
\end{align}
and
\begin{align}
&r_{1,j}^n = \paren{ \frac{\tau}{2}\p_t^2 c + \frac{h^2}{6}\p_x^3 c - \frac{h^2}{12}\p_x^4 c } (x_j,t^n) + \mc O(\tau^2 + h^4),\\
&r_{2,j(n)}^n = \frac{1}{6} \p_x^3 \Corr(\gamma^n,t^n) \paren{1 -\frac{vh}{2}} \paren{1-\sigma^n}^3 h + \mc O(h^2) , \\
&r_{2,j(n)+1}^n = \frac{1}{6}\p_x^3 \Corr(\gamma^n,t^n)\paren{1 + \frac{vh}{2 }} \paren{\sigma^n}^3 h + \mc O(h^2),\\
&r_{3,j(n)}^n = \frac{1}{\tau}\delta_{j(n-1), j(n)-1}\paren{\frac{1}{6} \p_x^3 \Corr(\gamma^{n-1},t^{n-1}) \paren{1-\sigma^{n-1}}^3h^3 + \mc O\paren{ h^4 } }, \\
&r_{3,j(n)+1}^n = \frac{1}{\tau} \delta_{j(n-1),j(n)+1}\paren{\frac{1}{6}  \p_x^3 \Corr(\gamma^{n-1},t^{n-1}) \paren{\sigma^{n-1}}^3h^3  + \mc O\paren{ h^4 } },\\
&r_4^n = \paren{ \frac{\textcolor{blue}{1-3\paren{\sigma^n}^2}}{6} \p_x^3 c +\alpha^n\frac{\sigma^n\paren{1-\sigma^n}}{2} \p_{xx}c } h^2\\
&\quad\ - \paren{ \frac{\paren{1-\sigma^n}^3}{6}  \paren{\sigma^n+\frac{1}{2}+\sigma^n\alpha^nh}\p_x^3 \Corr} h^2 + \mc O(h^3).
\end{align}
\end{lemma}
\begin{proof}
Define
\begin{align}
    \delta \bm C^{n-1} &= \delta_{j(n-1), j(n)-1} \Corr(x_{j(n)}, t^{n-1})  e_{j(n)} - \delta_{j(n-1),j(n)+1}\Corr(x_{j(n)+1}, t^{n-1}) e_{j(n)+1},\\
    \delta \bm d^{n-1} &= \delta_{j(n-1), j(n)-1} d_{j(n)}^{n-1}  e_{j(n)} - \delta_{j(n-1),j(n)+1}d_{j(n)+1}^{n-1} e_{j(n)+1},\\
    \delta \bm c^{n-1} &= \delta_{j(n-1), j(n)-1} \Corr_{j(n)}^{n-1}  e_{j(n)} - \delta_{j(n-1),j(n)+1}\Corr_{j(n)+1}^{n-1} e_{j(n)+1},
\end{align}
where $d_j^n$ is the truncated expansion of the exact correction function, defined as
\begin{equation}
\begin{aligned}
    d^n_j = \paren{x_j-\gamma^n} \paren{ 1-\frac{1}{2} \alpha^n\paren{x_j-\gamma^n}}\psi(t^n) - \frac{1}{2}\paren{x_j-\gamma^n}^2 f(\gamma^n, t^n).
\end{aligned}
\end{equation}
We split the local truncation error into three terms:
\begin{equation}
\begin{aligned}
    & \frac{\bm C^n - \paren{\bm C^{n-1} + h Z^{n-1} \psi(t^{n-1})} }{\tau} + \paren{\frac{1}{h^2}A + \frac{v}{2h}B} \bm C^n - \frac{1}{h}E^n \psi(t^n) -\wt{\bm f}^n\\
    &= \bm r_1^n + \bm r_2^n + \bm r_3^n
\end{aligned}
\end{equation}
where
\begin{align}
\bm r_1^n & = \frac{\bm C^n - \paren{\bm C^{n-1} + \delta \bm C^{n-1} } }{\tau} + \paren{\frac{1}{h^2}A + \frac{v}{2h}B} \bm C^n - \bm f^n \nonumber\\
& - \paren{\frac{1}{h^2} -\frac{v}{2h}}\Corr(x_{j(n)+1}, t^n) e_{j(n)}  + \paren{\frac{1}{h^2}+\frac{v}{2h}} \Corr(x_{j(n)}, t^n) e_{j(n)+1},\\
\bm r_2^n &= \paren{\frac{1}{h^2} -\frac{v}{2h}}\paren{\Corr(x_{j(n)+1}, t^n) - d_{j(n)+1}^n} e_{j(n)}\nonumber \\
&- \paren{\frac{1}{h^2} + \frac{v}{2h}} \paren{ \Corr(x_{j(n)}, t^n) - d_{j(n)}^n} e_{j(n)+1} = r_{2,j(n)}^n e_{j(n)} + r_{2,j(n)+1}^n e_{j(n)+1} ,\\
\bm r_3^n &= \textcolor{blue}{\frac{\delta \bm C^{n-1} - \delta \bm d^{n-1}}{\tau}} = r_{3,j(n)}^n e_{j(n)} + r_{3,j(n)+1}^n e_{j(n)+1} .
\end{align}
The boundary-condition residual is
\begin{align}
    r_4^n &
    = \frac{1}{h} (D^n)^T \bm C^n + \frac{D^n_{j(n)+1}}{h} \Corr(x_{j(n)+1},t^n) - \textcolor{blue}{g^n} + \frac{1}{h}D^n_{j(n)+1}\paren{ d_{j(n)+1}^n - \Corr(x_{j(n)+1},t^n) }.
\end{align}
At regular nodes the stencil does not cross $\gamma^n$. Expanding
$c(x_j,t^{n-1})$ about $(x_j,t^n)$ and $c(x_{j\pm1},t^n)$ about
$(x_j,t^n)$ gives
\begin{equation}
\begin{aligned}
    &\frac{c(x_j,t^n)-c(x_j,t^{n-1})}{\tau}
    +v\frac{c(x_{j+1},t^n)-c(x_{j-1},t^n)}{2h}\\
    &\quad-\frac{c(x_{j+1},t^n)-2c(x_j,t^n)+c(x_{j-1},t^n)}{h^2}
    -f(x_j,t^n),
\end{aligned}
\end{equation}
which yields the stated formula for $r_{1,j}^n$ after using the PDE.
At the two irregular nodes, the exact correction has the Taylor expansion
\begin{equation}
\begin{aligned}
    \Corr(x_{j(n)+1},t^n)-d_{j(n)+1}^n
    &=\frac16 \p_x^3\Corr(\gamma^n,t^n)(1-\sigma^n)^3h^3+\mc O(h^4),\\
    \Corr(x_{j(n)},t^n)-d_{j(n)}^n
    &=-\frac16 \p_x^3\Corr(\gamma^n,t^n)(\sigma^n)^3h^3+\mc O(h^4).
\end{aligned}
\end{equation}
Substitution into the two correction terms in $\bm r_2^n$ gives the stated
$r_{2,j(n)}^n$ and $r_{2,j(n)+1}^n$. The vector $\bm r_3^n$ is present only
when the interface crosses a grid node between $t^{n-1}$ and $t^n$; the same
Taylor expansion at time $t^{n-1}$ gives the two listed crossing residuals.
Finally, expanding $c(x_{j(n)-1},t^n)$, $c(x_{j(n)},t^n)$, and
$c(x_{j(n)+1},t^n)+\Corr(x_{j(n)+1},t^n)$ about $\gamma^n$ in the discrete
Robin equation~\eqref{eqn:bc-discrete}, using
$c_x+\alpha c=g$ and $\Corr(\gamma^n,t^n)=0$, gives the displayed expression
for $r_4^n$. The remainders are uniform by the assumed regularity of $c$ and
$\Corr$ in the narrow band.
\end{proof}

Using the same reduction as in Proposition~\ref{prop:red-scheme}, we obtain the reduced error equation.
\begin{corollary}\label{cor:red-LTE}
For $n=1,2,\cdots, N_t$, the error $\bm\varepsilon_c^n$ satisfies
\begin{equation}\label{eqn:err-eqn}
\begin{aligned}
    & R^n K^n (R^n)^T R^n \bm\varepsilon_c^n - R^n M^{n-1}  (R^{n-1})^T R^{n-1} \bm \varepsilon_c^{n-1} \\
    &= R^n \tau\paren{ \sum_{k=1}^3 \bm r_k^n+ \frac{1}{hQ^n} E^n  r_4^n} + \frac{h}{Q^{n-1}} R^n  Z^{n-1} r_4^{n-1}.
\end{aligned}
\end{equation}
\end{corollary}

Since $R^n\bm r_2^n=r_{2,j(n)}^n e_{j(n)}$ and $R^nE^n=E_{j(n)}^n e_{j(n)}$, the interface residual and the boundary residual can be factored through the same active-grid vector. Thus we can rewrite the error equation as
\begin{equation}\label{eq:consistency}
\begin{aligned}
    & R^n K^n (R^n)^T R^n \bm\varepsilon_c^n - R^n M^{n-1}  (R^{n-1})^T R^{n-1} \bm \varepsilon_c^{n-1}     \\
    &= \tau R^n \bm r_1^n + \tau R^n\paren{ \bm r_2^n+ \frac{1}{h Q^n} E^n  r_4^n} + \tau R^n \bm r_3^n  +  \frac{h}{Q^{n-1}} R^n Z^{n-1} r_4^{n-1} \\
     &= \underbrace{\tau R^n \bm r_1^n}_{\mc O(\tau\paren{\tau+h^2})} + \underbrace{\paren{\frac{h Q^n  r_{2,j(n)}^n}{E_{j(n)}^n} + r_4^n} \frac{ \tau}{h Q^n}R^nE^n}_{\mc O(\tau h)} + \underbrace{\tau R^n \bm r_3^n  +  \frac{h}{Q^{n-1}} R^n Z^{n-1} r_4^{n-1}}_{\mc O(h^3)}.
\end{aligned}
\end{equation}
That the middle term has order $\mc{O}(\tau h)$ can be seen as follows. For sufficiently small $h$, $Q^n$ and $E_{j(n)}^n$ are bounded away from zero and $R^nE^n/Q^n$ is uniformly bounded. Since $r_{2,j(n)}^n=\mc{O}(h)$ and $r_4^n=O(h^2)$, the bracketed scalar is $\mc{O}(h^2)$ and the displayed vector has $\ell^\infty$ size of $\mc{O}(\tau h)$. 

Since the number of time steps is of order $1/\tau$, the above consistency estimate implies that the error $\varepsilon_c^n$ should scale like
\begin{equation}
\varepsilon_c^n =\mc{O}(\tau+h^2)+\mc{O}(h)+\mc{O}(h^3/\tau).
\end{equation}
This is considerably worse than the computationally observed error rate of $\mc{O}(\tau+h^2)$ (see Section \ref{sect:1Dconv}). We now establish two lemmas that help resolve this issue. 

\textcolor{blue}{The $\mc{O}(h^3)$ term in the last line of} \eqref{eq:consistency}
\textcolor{blue}{arises only when the moving boundary crosses a grid point. This can be quantified with the help of the following moving-interface summation bound.}

\begin{lemma}\label{lem:delta-sum}
\textcolor{blue}{Assume that either $\dot\gamma\ge 0$ on $[0,T]$ or
$\dot\gamma\le 0$ on $[0,T]$. Then, for $n=1,2,\cdots,N_t$,}
\begin{equation}
    \textcolor{blue}{\sum_{\ell=1}^n\delta_{j(\ell-1), j(\ell)-1}
    \le 1+\frac{1}{h}\int_0^{t^n} \abs{\dot{\gamma}(s)}\,ds.}
\end{equation}
\end{lemma}
\begin{proof}
\textcolor{blue}{If $\dot\gamma\le0$, the sum is zero. If $\dot\gamma\ge0$,
then $j(\ell)-j(\ell-1)\in\{0,1\}$ and hence}
\begin{equation}
    \textcolor{blue}{\sum_{\ell=1}^n\delta_{j(\ell-1),j(\ell)-1}
    =j(n)-j(0)
    \le 1+\frac{\gamma(t^n)-\gamma(0)}{h}
    =1+\frac{1}{h}\int_0^{t^n}\abs{\dot\gamma(s)}\,ds.}
\end{equation}
\end{proof}

The middle term of $\mc{O}(\tau h)$ in the last line of \eqref{eq:consistency} is localized near the moving boundary. We construct a modified solution that cancels this term while remaining $\mc{O}(h^2)$ close to the original numerical solution.

\begin{lemma}\label{lem:mod-LTE}
There exists a grid function $R^n\wh{\bm c}^n$, with
$R^0\wh{\bm c}^0=0$, such that
\begin{equation}
    \norm{R^n\wh{\bm c}^n}_\infty
    \le Ce^{CT+\beta L}h^2 .
\end{equation}
Defining the modified error by
$R^n\wt{\bm\varepsilon}_c^n=R^n\bm\varepsilon_c^n-R^n\wh{\bm c}^n$, the modified error satisfies
\begin{align}
    & R^n K^n (R^n)^T R^n \wt{\bm\varepsilon}_c^n - R^n M^{n-1}  (R^{n-1})^T R^{n-1} \wt{\bm\varepsilon}_c^{n-1} = \tau R^n\wt{\bm r}_1^n + R^n\wt{\bm r}_2^n,
\end{align}
and
\begin{align}\label{eqn:wt_LTE_bound}
    \norm{R^n\wt{\bm r}_1^n}_\infty \le C\paren{\tau +h^2},\quad
    \norm{R^n\wt{\bm r}_2^n}_\infty
    \le \textcolor{blue}{C\delta_{j(n-1),j(n)-1}h^3}.
\end{align}
\end{lemma}
\begin{proof}
Define $R^n\wh{\bm c}^n$ as the solution of
\begin{equation}
R^n K^n (R^n)^T R^n \wh{\bm c}^n - R^n M^{n-1}  (R^{n-1})^T R^{n-1} \wh{\bm c}^{n-1} =  \frac{ \tau}{h Q^n}R^nE^n \wh g^n,
\end{equation}
where $\wh g^n$ is given by
\begin{equation}
    \wh g^n = \paren{\frac{h Q^n  r_{2,j(n)}^n}{E_{j(n)}^n} + r_4^n} .
\end{equation}
The definition of $\wh g^n$ cancels the concentrated term
$\tau R^n(\bm r_2^n+E^nr_4^n/(hQ^n))$ in
Corollary~\ref{cor:red-LTE}.
\textcolor{blue}{Set $R^n\wt{\bm r}_1^n=R^n\bm r_1^n$ and
$R^n\wt{\bm r}_2^n=\tau R^n\bm r_3^n+
hR^nZ^{n-1}r_4^{n-1}/Q^{n-1}$. Write}
\begin{align*}
&\textcolor{blue}{\theta^{n-1}
=\sigma^{n-1}+\frac12+\sigma^{n-1}\alpha^{n-1}h,}\\
&\textcolor{blue}{G^{n-1}
=\frac{1-3\paren{\sigma^{n-1}}^2}{6}
\p_x^3c(\gamma^{n-1},t^{n-1})
+\alpha^{n-1}\frac{\sigma^{n-1}\paren{1-\sigma^{n-1}}}{2}
\p_{xx}c(\gamma^{n-1},t^{n-1}).}
\end{align*}
\textcolor{blue}{Then
$R^nZ^{n-1}/Q^{n-1}
=\delta_{j(n-1),j(n)-1}e_{j(n)}/\theta^{n-1}$.}
\textcolor{blue}{The formulas for $r_3^n$ and $r_4^{n-1}$ give remainders
$\eta_3^n$ and $\eta_4^{n-1}$ satisfying
$\abs{\eta_3^n}\le Ch^4$ and $\abs{\eta_4^{n-1}}\le Ch^3$, such that}
\begin{align*}
&\textcolor{blue}{R^n\wt{\bm r}_2^n
=\delta_{j(n-1),j(n)-1}\left[
\frac16\p_x^3\Corr(\gamma^{n-1},t^{n-1})
\paren{1-\sigma^{n-1}}^3h^3+\eta_3^n\right.}\\
&\textcolor{blue}{\left.\hspace{2.8cm}
+\frac{h}{\theta^{n-1}}\left(
G^{n-1}h^2
-\frac16\paren{1-\sigma^{n-1}}^3\theta^{n-1}
\p_x^3\Corr(\gamma^{n-1},t^{n-1})h^2
+\eta_4^{n-1}\right)\right]e_{j(n)}}\\
&\textcolor{blue}{\hspace{1.8cm}
=\delta_{j(n-1),j(n)-1}\left[
\frac{h^3}{\theta^{n-1}}G^{n-1}
+\eta_3^n+\frac{h}{\theta^{n-1}}\eta_4^{n-1}\right]e_{j(n)}.}
\end{align*}
\textcolor{blue}{Since $\theta^{n-1}$ is bounded away from zero for
sufficiently small $h$,}
\begin{equation*}
\textcolor{blue}{\norm{R^n\wt{\bm r}_2^n}_\infty
\le C\delta_{j(n-1),j(n)-1}h^3.}
\end{equation*}
\textcolor{blue}{This proves the second bound in}
\eqref{eqn:wt_LTE_bound}\textcolor{blue}{; the first follows directly from
Lemma~\ref{lem:LTE}.}
Since $\abs{\wh g^n} \le Ch^2$, applying the discrete maximum principle in Theorem~\ref{thm:mp-o2} to $\pm\wh{\bm c}^n$ gives
\begin{equation}
    \norm{R^n\wh{\bm c}^n}_\infty \le e^{CT+\beta L}\sup_{0\le n\le N_t}\abs{\wh g^n} \le Ce^{CT+\beta L} h^2.
\end{equation}

\end{proof}

We may now prove our error estimate.

\begin{theorem}\label{thm:err-est-o2}
Consider the one-dimensional problem~\eqref{eqn:prob-1d} and the scheme
\eqref{eqn:scheme}. \textcolor{blue}{Let $c$ and the correction function
$\Corr$ have the regularity used in Lemma}~\ref{lem:LTE}\textcolor{blue}{,
and let $\dot\gamma$, $\alpha$, and the data be uniformly bounded on
$[0,T]$. Assume that either $\dot\gamma\ge 0$ on $[0,T]$ or
$\dot\gamma\le 0$ on $[0,T]$, and that the compatibility conditions needed
for the continuous maximum principle in Theorem~\ref{thm:max_principle} hold after
the exponential transform when $\alpha$ is not positive. Initialize the scheme
exactly, $R^0\bm c^0=R^0\bm C^0$, and let the mesh parameters satisfy}
\begin{equation}
    \frac{1}{2}h^2 \le \tau < \frac{h}{\max_{[0,T]}\abs{\dot{\gamma}(t)}} .
\end{equation}
Then there exists a constant $h_0$ depending on $\gamma$, $v$, and $\alpha$
such that for all $h\le h_0$, the numerical solution
$R^n\bm c^n$ satisfies
\begin{equation}\label{eqn:err-est-o2}
\begin{aligned}
    \sup_{0\le n\le N_t}\norm{R^n\bm\varepsilon_c^n}_\infty  \le  C e^{CT+\beta L} \paren{\tau +h^2}.
\end{aligned}
\end{equation}
This theorem applies only to the one-dimensional scheme analyzed in this section.
\end{theorem}
\begin{proof}
Using Lemma~\ref{lem:mod-LTE},
\begin{equation}
\begin{aligned}
    R^nW^{-1}\wt{\bm\varepsilon}_c^n
    &= (K_I^n)^{-1}T^n R^{n-1}W^{-1}\wt{\bm\varepsilon}_c^{n-1}
    +  (K_I^n)^{-1} R^nW^{-1}\paren{\tau\wt{\bm r}_1^n + \wt{\bm r}_2^n}, \\
    &= \paren{\prod_{\ell=1}^n (K_I^\ell)^{-1}T^\ell }R^0 W^{-1}\wt{\bm\varepsilon}_c^0 \\
    &+  \sum_{m=1}^n\paren{\prod_{\ell=m+1}^n (K_I^\ell)^{-1}T^\ell }R^m W^{-1}\paren{\tau\wt{\bm r}_1^m + \wt{\bm r}_2^m}.
\end{aligned}
\end{equation}
Since $\wt{\bm\varepsilon}_c^0=0$, Lemmas~\ref{lem:KM-mat-o2} and~\ref{lem:delta-sum} give
\begin{equation}
\begin{aligned}
    \norm{R^nW^{-1}\wt{\bm\varepsilon}_c^n}_\infty & \le \frac{5}{3} \sum_{\ell=1}^n \paren{1+C\tau}^{n-\ell}  \paren{\tau \norm{R^\ell\wt{\bm r}_1^\ell}_\infty + \norm{R^\ell\wt{\bm r}_2^\ell}_\infty} \\
    & \le C \paren{e^{Cn\tau}-1}\paren{\tau +h^2} + Ce^{Cn\tau}
    \textcolor{blue}{\sum_{\ell=1}^n \delta_{j(\ell-1), j(\ell)-1} h^3}\\
    & \le C e^{Ct^n} \paren{\tau +h^2} .
\end{aligned}
\end{equation}
Also,
\begin{equation}
\begin{aligned}
    \norm{R^n\bm\varepsilon_c^n}_\infty
    \le \norm{R^nW W^{-1}\wt{\bm\varepsilon}_c^n}_\infty
    + \norm{R^n\wh{\bm c}^n}_\infty
    \le  C e^{Ct^n+\beta L}  \paren{\tau +h^2},\quad n = 1,2,\cdots, N_t,
\end{aligned}
\end{equation}
Taking the maximum over all $n$ gives~\eqref{eqn:err-est-o2}.
\end{proof}

\begin{remark}
\textcolor{blue}{The convergence result in Theorem~\ref{thm:err-est-o2} remains valid if the interface changes direction only
finitely many times. Indeed, one may partition $[0,T]$ into finitely many
intervals on which the interface motion is monotone and apply Lemma~
\ref{lem:delta-sum} on each interval. Summing the resulting
bounds changes only the constant in the error estimate.}
\end{remark}

\begin{remark}
The convergence analysis above applies only to the one-dimensional scheme in
this section. The two-dimensional method in Section~\ref{sec:method} uses the
same correction idea, but the proof is harder because the geometry and grid
crossings are more complicated. We therefore check the two-dimensional
convergence numerically in Section~\ref{sec:results}.
\end{remark}

\hypersetup{linkcolor=black,citecolor=black,urlcolor=black}

\section{Numerical results}\label{sec:results}

This section presents numerical results for the proposed scheme.
We first check the one-dimensional problem analyzed in
Section~\ref{sec:numerical-analysis}. We then test the two-dimensional method
on manufactured solutions and on an active flux-driven problem without a known
exact solution. All errors are measured in discrete $\ell^\infty$ norms, either
on Cartesian cell centers or on interface marker points. The computations were
performed using code written in Fortran 90 and C++ on Dell PowerEdge R7525
(AMD EPYC 7H12 processors 2.60 GHz). Unless stated otherwise, the matrix-free
GMRES iteration is stopped when the relative residual is below $10^{-6}$.

\subsection{Convergence study in one dimension}\label{sect:1Dconv}
We apply the scheme to~\eqref{eqn:prob-1d} on
$\mc B=(0,2)$. The moving boundary is
$\gamma(t)=1+0.2\sin(t)$, and the background velocity is $v=1/2$. The exact
solution is
\begin{equation}
    c_{\rm ex}(x,t)=\sin(\pi x/2)e^{-t}.
\end{equation}
It satisfies~\eqref{eqn:prob-1d} with source term
\begin{equation}
    f(x,t) = e^{-t} \left[ \left( \frac{\pi^2}{L^2} - 1 \right) \sin\left(\frac{\pi x}{L}\right) + \frac{v \pi}{L} \cos\left(\frac{\pi x}{L}\right) \right].
\end{equation}
At the moving interface $x=\gamma(t)$, the Robin data are
\begin{equation}
    g(t) = e^{-t} \left[ \frac{\pi}{L} \cos\left(\frac{\pi \gamma(t)}{L}\right) + \alpha(t) \sin\left(\frac{\pi \gamma(t)}{L}\right) \right],
\end{equation}
where $\alpha(t)=\dot{\gamma}(t)-v$.

The error is the space--time $\ell^\infty$ error up to
$T=0.5$. Table~\ref{tab:conv-1d} contains three refinements. Case 1 uses the
parabolic scaling $\tau=0.5h^2$. Case 2 fixes $\tau=10^{-5}$ and refines the
grid. Case 3 fixes $N=400$ and refines the time step. The results show
second-order convergence in space and first-order convergence in time. In
Case 1, $\tau/h^2=0.5$ at every level, which is the lower endpoint of the range
used in the proof.

\begin{table}[htbp]
\centering
\caption{Convergence for the one-dimensional manufactured
solution. Errors are measured in the space--time $\ell^\infty$ norm up to
$T=0.5$.}
\label{tab:conv-1d}
\begin{tabular}{c c c c c}
\toprule
\multicolumn{5}{l}{Case 1: Parabolic scaling ($\tau = 0.5 h^2$)} \\
\midrule
$N$ & $h$ & $\tau$ & $\ell^\infty$ error & Order \\
\midrule
20  & 0.1000 & 5.00e-03 & 2.59e-03 & ---  \\
40  & 0.0500 & 1.25e-03 & 6.93e-04 & 1.90 \\
80  & 0.0250 & 3.13e-04 & 1.80e-04 & 1.94 \\
160 & 0.0125 & 7.81e-05 & 4.61e-05 & 1.97 \\
\midrule
\multicolumn{5}{l}{Case 2: Spatial refinement (fixed $\tau = 10^{-5}$)} \\
\midrule
$N$ & $h$ & & $\ell^\infty$ error & Order \\
\midrule
20  & 0.1000 & & 2.08e-03 & ---  \\
40  & 0.0500 & & 5.45e-04 & 1.93 \\
80  & 0.0250 & & 1.42e-04 & 1.94 \\
160 & 0.0125 & & 3.72e-05 & 1.93 \\
\midrule
\multicolumn{5}{l}{Case 3: Temporal refinement (fixed $N=400$)} \\
\midrule
$\tau$ & & & $\ell^\infty$ error & Order \\
\midrule
0.01000 & & & 1.32e-03 & ---  \\
0.00500 & & & 6.66e-04 & 0.99 \\
0.00250 & & & 3.36e-04 & 0.99 \\
0.00125 & & & 1.70e-04 & 0.98 \\
\bottomrule
\end{tabular}
\end{table}


\subsection{Convergence study with manufactured solutions}

In the two-dimensional tests, $N$ denotes the number of
Cartesian cells in each coordinate direction, $N_t$ denotes the number of time
steps, and $N_\Gamma$ denotes the number of interface marker points. Errors are
measured in the discrete $\ell^\infty$ norm on the indicated cell centers or
interface points. In the mesh-refinement studies, $\tau=1/N_t$ is refined
together with the Cartesian mesh and the interface discretization. We first check the convergence of the method on problems with manufactured solutions.

\subsubsection{Circular interface}
We first consider an interior-domain test in the box $\mc B=(-1,1)^2$. The
physical domain $\Omega(t)$ is the region enclosed by a moving interface
$\Gamma(t)$. Initially, $\Gamma(0)$ is parametrized by
\begin{equation}
(x(s),y(s))=(-0.2+0.6\cos s,\,0.6\sin s),
\qquad 0\leq s<2\pi.
\end{equation}
The prescribed velocity field is
\begin{equation}
\bm u(x, y, t) = \left(\frac{1}{10} (1-y^2),\,0\right),
\end{equation}
and the interface is transported by the same velocity,
$\p_t\bm X(s,t)=\bm u(\bm X(s,t),t)$. The manufactured concentration is
\begin{equation}
c_{\rm ex}(x, y, t) = \sin(t) \cos(\pi x) \sin(\pi y).
\end{equation}
The source term and the Robin data on $\Gamma(t)$ are obtained by substituting $c_{\rm ex}$ into the continuous problem. The computation is run to $T=1$. For a grid with $N$ cells in each coordinate direction, $h=2/N$ and $\tau=1/N_t$. When $N$ is doubled, $N_t$ is increased by
a factor of four and $N_\Gamma$ by a factor of two; hence $\tau=\mc O(h^2)$ and
the interface spacing is refined at the same rate as the Cartesian mesh. Errors are measured for both cell centers in $\Omega(t)$ and interface marker points.




Figure~\ref{fig:circle-snapshots} shows snapshots of the numerical solution on
the moving interior domain. The error curves in Figure~\ref{fig:circle-error-h} show
approximately second-order decay for both the bulk concentration and the
interface trace along the parabolic refinement path $\tau=\mc O(h^2)$. The two
curves are close to one another, indicating that the boundary trace is as accurate as the bulk solution, which is essential for problems where the boundary dynamics couple with the concentration.
Figure~\ref{fig:circle-gmres} reports the GMRES iteration counts for the
interface-density solve. The counts remain essentially mesh independent, which implies the system~\eqref{eqn:dis-bie} is well-conditioned.

\begin{figure}[htbp]
    \centering
    \includegraphics[width=0.8\textwidth]{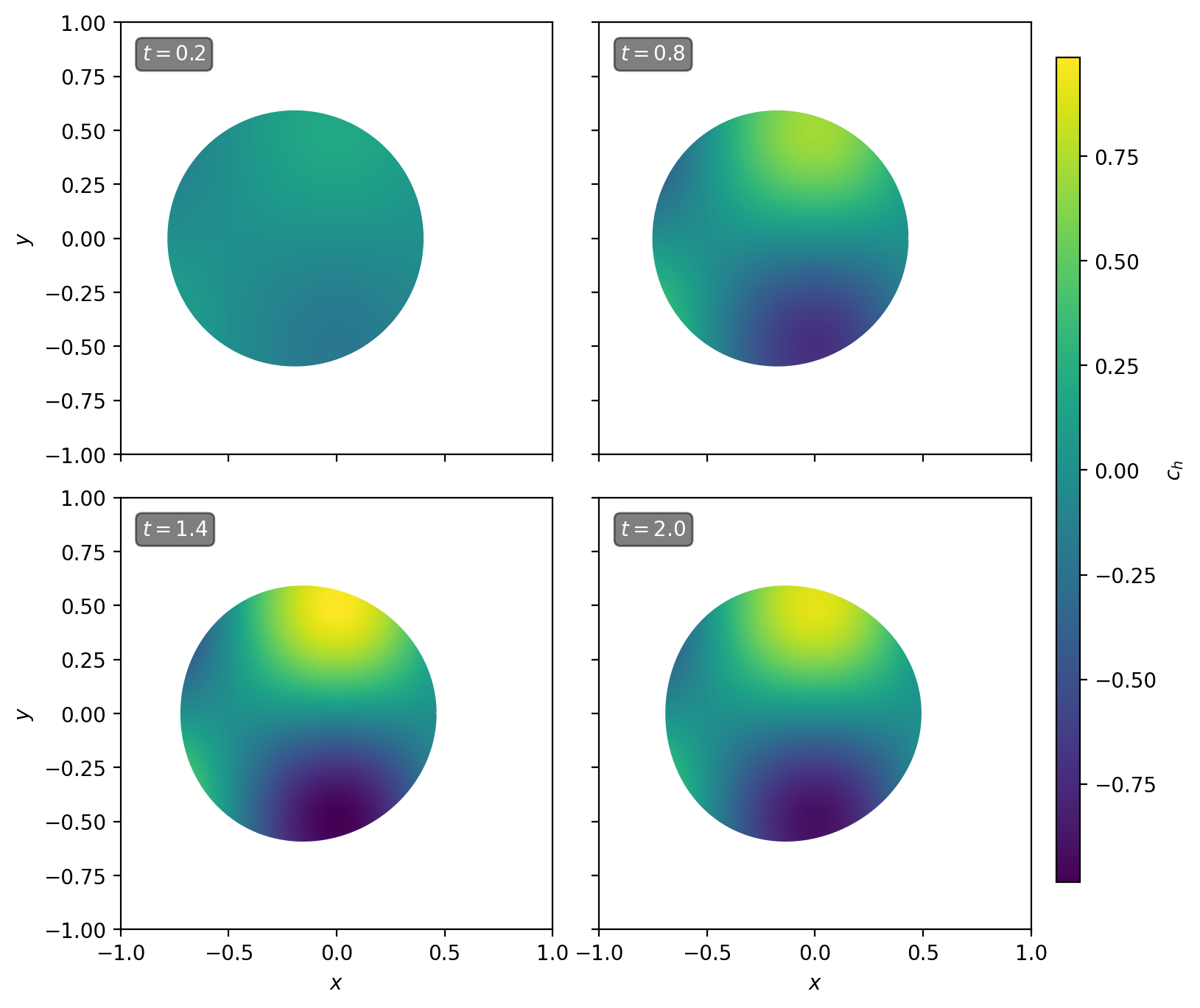}
    \caption{Snapshots of the manufactured concentration
    inside the circular moving interface.}
    \label{fig:circle-snapshots}
\end{figure}

\begin{figure}[htbp]
    \centering
    \includegraphics[width=0.7\textwidth]{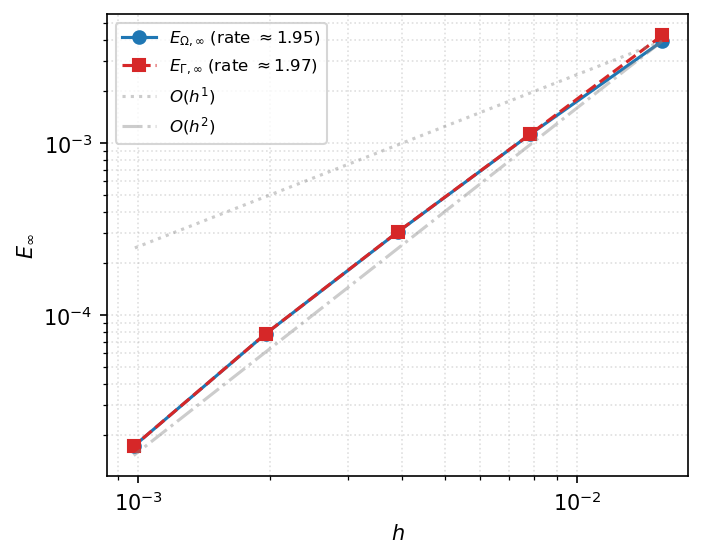}
    \caption{Space--time maximum errors for the bulk
    concentration and interface trace in the circular-interface test.}
    \label{fig:circle-error-h}
\end{figure}

\begin{figure}[htbp]
    \centering
    \includegraphics[width=0.7\textwidth]{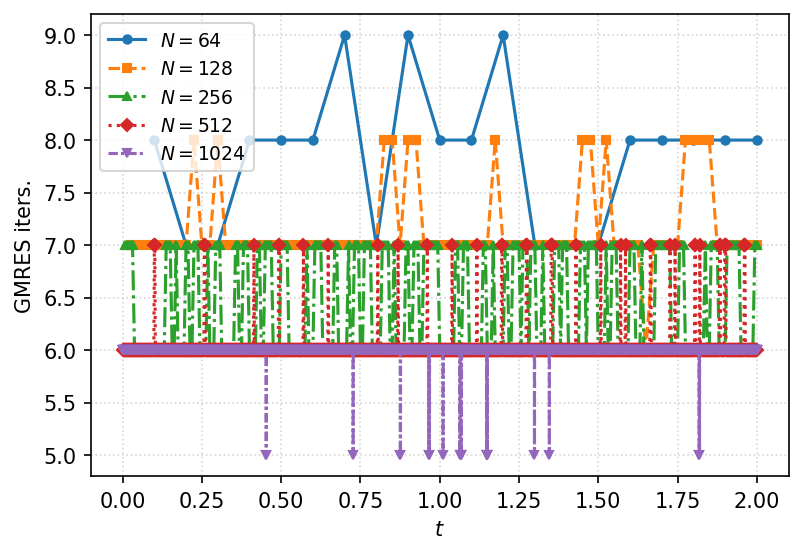}
    \caption{GMRES iteration counts for the circular-interface
    test.}
    \label{fig:circle-gmres}
\end{figure}

\subsubsection{Star-shaped interface}
The second example uses the same box, but the physical domain is the
exterior of a star-shaped moving interface inside $\mc B$. 

The velocity field and manufactured concentration are
\begin{equation}
\bm u(x, y, t) = \left(\frac{1}{10} (1-y^2),\,0\right),
\end{equation}
and
\begin{equation}
c_{\rm ex}(x, y, t) = \sin(t) \cos(\pi x) \sin(\pi y).
\end{equation}
As before, the interface is advected by the prescribed flow,
$\p_t\bm X(s,t)=\bm u(\bm X(s,t),t)$, and the source term is obtained by
substituting $c_{\rm ex}$ into the PDE.
Dirichlet data on $\partial\mc B$ and Robin data on $\Gamma(t)$ are taken from the manufactured solution.

The initial interface is defined by the parametric curve
\begin{equation}
(x(s), y(s)) =
\left(
\frac{2\pi(5.5+\cos(5s))\cos s}{64}-0.2,
\frac{2\pi(5.5+\cos(5s))\sin s}{64}
\right),
\end{equation}
where $s$ is sampled uniformly in $[0,2\pi]$. Compared with the circular
case, this geometry has larger curvature variation and produces a less regular
sequence of grid-interface intersections. 

The computation is again run to $T=1$ with $h=2/N$ and $\tau=1/N_t$. The same
refinement sequence for $N$, $N_t$, and $N_\Gamma$ is used as in the circular
test. Errors are measured at cell centers in the exterior domain and interface marker points.




Figure~\ref{fig:star-snapshots} shows snapshots of the numerical solution on
the exterior moving domain. The convergence plot in
Figure~\ref{fig:star-error-h} shows the same qualitative behavior as in the
circular test: both the bulk error and the interface-trace error decrease at a rate close to second-order once the interface is sufficiently resolved. The boundary error remains comparable to the bulk error. Figure~\ref{fig:star-gmres} shows that the GMRES iteration counts remain bounded under refinement, even for this
less regular moving interface.

\begin{figure}[htbp]
    \centering
    \includegraphics[width=0.84\textwidth]{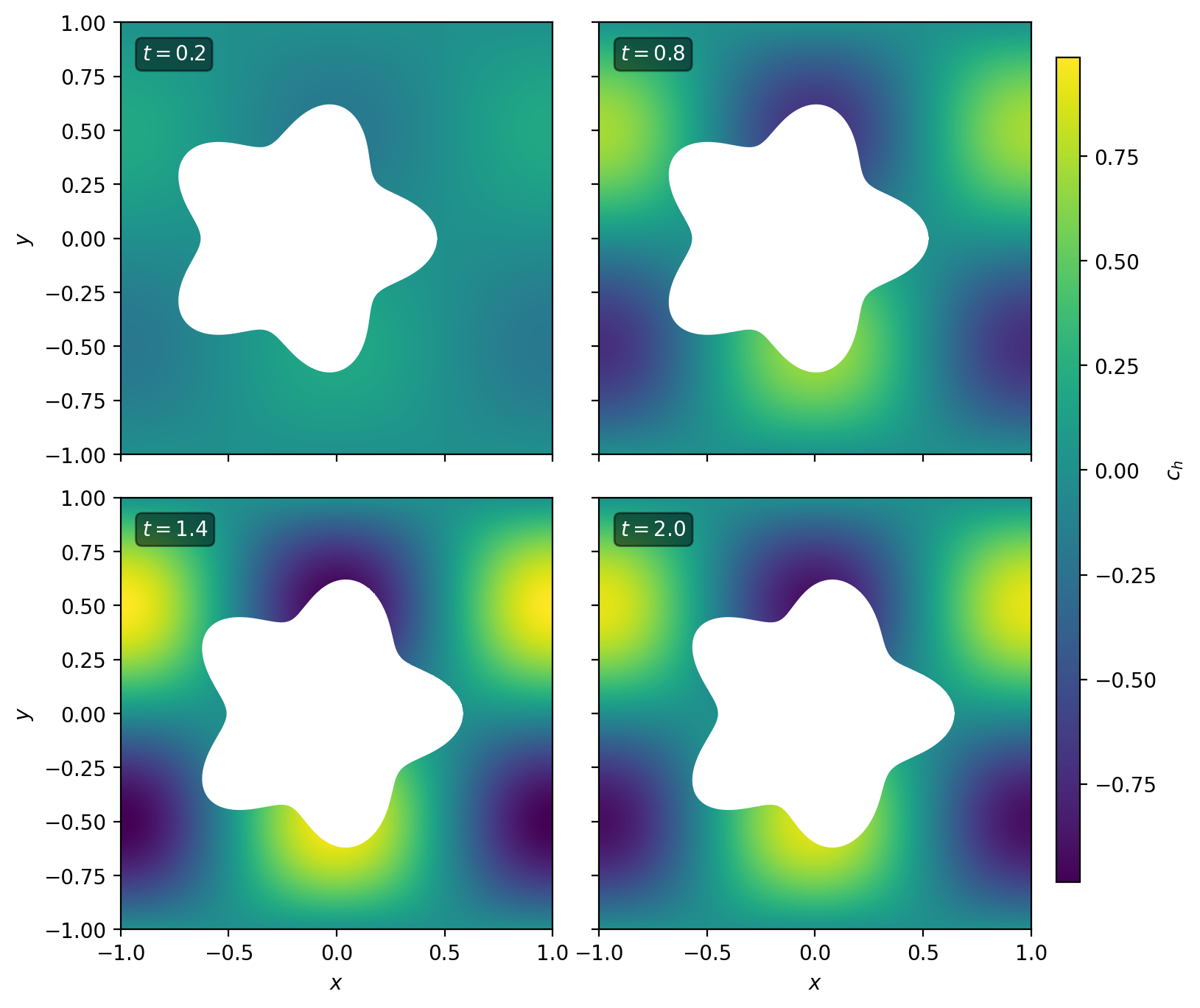}
    \caption{Snapshots of the manufactured concentration
    outside the star-shaped moving interface.}
    \label{fig:star-snapshots}
\end{figure}

\begin{figure}[htbp]
    \centering
    \includegraphics[width=0.7\textwidth]{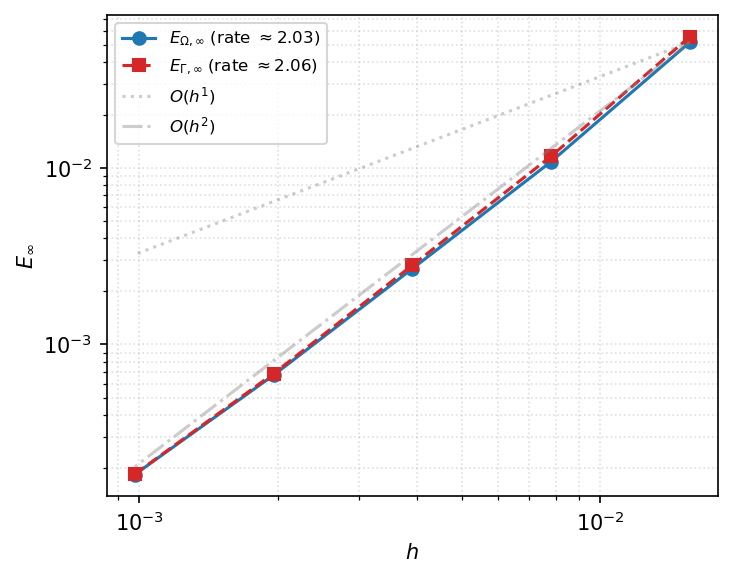}
    \caption{Space--time maximum errors for the bulk
    concentration and interface trace in the star-shaped exterior test.}
    \label{fig:star-error-h}
\end{figure}

\begin{figure}[htbp]
    \centering
    \includegraphics[width=0.7\textwidth]{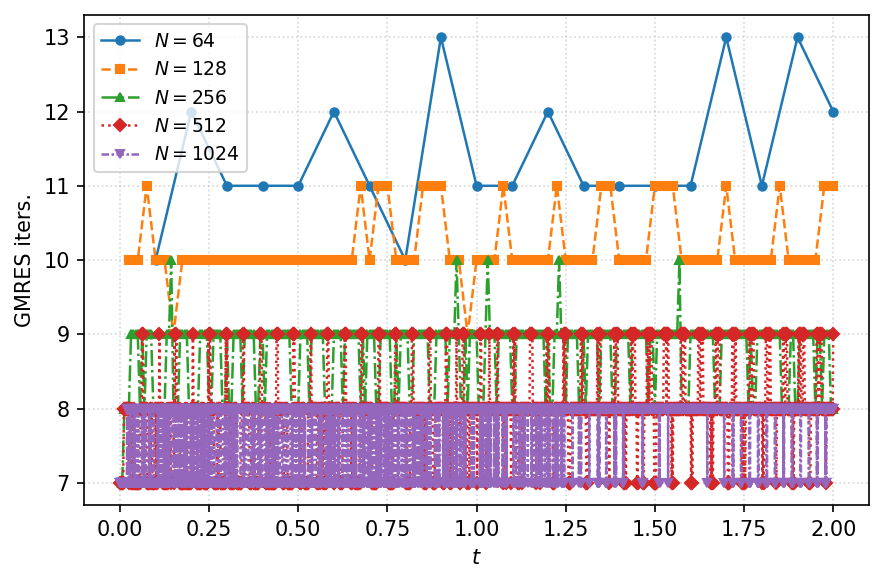}
    \caption{GMRES iteration counts for the star-shaped
    exterior test.}
    \label{fig:star-gmres}
\end{figure}

\subsection{Coupled interface dynamics}

\subsubsection{Problem setup}
We next consider a coupled transport problem motivated by the osmotic-flow
model of Yao and Mori~\cite{Yao2017}. The initial interface is the circle
centered at the origin with radius $0.6$, parametrized by
$\bm X(s,0)=0.6(\cos s,\sin s)$, $0\leq s<2\pi$.
The initial concentration is $c_0=1$ on the box $\mc B=(-1,1)^2$.
No-flux boundary conditions are imposed on $\partial\mc B$. The concentration
is transported by the divergence-free velocity field
\begin{equation}
\bm u(x, y, t)
=0.4\left(
-y+\cos(\pi t)\cos(y)\sin(x),\,
x-\cos(\pi t)\sin(y)\cos(x)
\right).
\end{equation}
The interface is advected by the same velocity field.

The interfacial flux consists of passive and active contributions. Let
$\Gamma_{\rm i}$ and $\Gamma_{\rm e}$ denote the inner and outer sides of the
interface, separating the regions $\Omega_{\rm i}$ and $\Omega_{\rm e}$. On
either side of the interface, the boundary condition is
\begin{equation}
\paren{\paren{\bm u-\p_t\bm X}c-D\grad c}\cdot \bm n
=j_{\rm c}+j_{\rm p}
\quad\text{on } \Gamma_{\rm i}\text{ or }\Gamma_{\rm e}.
\label{ckbc}
\end{equation}
The passive flux is
\begin{equation}\label{jcexp}
j_{\rm c}=k_{\rm c}(c_{\rm i}-c_{\rm e}),
\end{equation}
and the active flux is
\begin{equation}\label{jpexp}
\begin{split}
j_{\rm p}&=\widehat{k}_{\rm p}H(s, c_{\rm i}, c_{\rm e}), \qquad
\widehat{k}_{\rm p}=k_{\rm p}(s),\\
H(s,c_{\rm i}, c_{\rm e})&=
\begin{cases}
c_{\rm i}, &\text{if } k_{\rm p}(s)\geq 0,\\
c_{\rm e}, &\text{if } k_{\rm p}(s)<0.
\end{cases}
\end{split}
\end{equation}
In the computations, $D=1$ and $k_{\rm c}=0.01$. The active
pumping rate is prescribed as a function of the Lagrangian coordinate $s$
along the interface by
\begin{equation}
k_{\rm p}(s)=
\tanh\left(\frac{s^4}{10}\right)
+\tanh\left(\frac{(2\pi-s)^4}{10}\right)
-\frac{3}{2}.
\end{equation}
No exact solution is available for this problem. The example tests the method
when interface motion, passive exchange, and spatially varying active pumping
are all present.

\subsubsection{Error and convergence results}
Since there is no closed-form solution, errors are estimated by successive
refinement. In Tables~\ref{tab:coupled-bulk-t1}--\ref{tab:coupled-interface-t3},
the listed values of $N$, $N_t$, and $N_\Gamma$ correspond to the coarser grid
in each refinement pair. The reported error is the difference between the
coarse-grid solution and the next finer solution, restricted to the coarse
degrees of freedom. Thus the tables report self-convergence errors.

The self-convergence rates are close to second-order in most cases. At
$t=1$, Tables~\ref{tab:coupled-bulk-t1} and
\ref{tab:coupled-interface-t1} give orders between $1.80$ and $2.21$. At
$t=3$, Tables~\ref{tab:coupled-bulk-t3} and
\ref{tab:coupled-interface-t3} show the same pattern, except for the finest
interior bulk comparison, where the order is $1.43$. The exterior bulk error
and both interface errors remain close to second-order.
Figure~\ref{fig:coupled-refinement-diff} shows the
successive-refinement differences over time. The differences decrease with
refinement over the full interval shown. Figures~\ref{fig:coupled-snapshots}
and~\ref{fig:coupled-interface-profiles} show representative bulk snapshots
and interface concentration profiles.

\begin{table}[htbp]
\centering
\caption{Self-convergence of the bulk
concentration at $t=1$.}
\label{tab:coupled-bulk-t1}
\begin{tabular}{@{}c c c c c c c@{}}
\toprule
 & & & \multicolumn{2}{c}{$c_{\rm i}$} & \multicolumn{2}{c}{$c_{\rm e}$} \\
\cmidrule(lr){4-5}\cmidrule(l){6-7}
$N$ & $N_t$ & $N_\Gamma$ & $\ell^\infty$ error & Order & $\ell^\infty$ error & Order \\
\midrule
64  & 100  & 160 & 8.25e-04 & ---  & 1.60e-03 & ---  \\
128 & 400  & 320 & 2.10e-04 & 1.98 & 3.97e-04 & 2.01 \\
256 & 1600 & 640 & 4.65e-05 & 2.17 & 1.05e-04 & 1.91 \\
\bottomrule
\end{tabular}
\end{table}

\begin{table}[htbp]
\centering
\caption{Self-convergence of the interface
concentration at $t=1$.}
\label{tab:coupled-interface-t1}
\begin{tabular}{@{}c c c c c c c@{}}
\toprule
 & & & \multicolumn{2}{c}{$c_{\rm i}$} & \multicolumn{2}{c}{$c_{\rm e}$} \\
\cmidrule(lr){4-5}\cmidrule(l){6-7}
$N$ & $N_t$ & $N_\Gamma$ & $\ell^\infty$ error & Order & $\ell^\infty$ error & Order \\
\midrule
64  & 100  & 160 & 1.36e-03 & ---  & 9.47e-04 & ---  \\
128 & 400  & 320 & 3.36e-04 & 2.02 & 2.10e-04 & 2.17 \\
256 & 1600 & 640 & 7.26e-05 & 2.21 & 6.04e-05 & 1.80 \\
\bottomrule
\end{tabular}
\end{table}

\begin{table}[htbp]
\centering
\caption{Self-convergence of the bulk
concentration at $t=3$.}
\label{tab:coupled-bulk-t3}
\begin{tabular}{@{}c c c c c c c@{}}
\toprule
 & & & \multicolumn{2}{c}{$c_{\rm i}$} & \multicolumn{2}{c}{$c_{\rm e}$} \\
\cmidrule(lr){4-5}\cmidrule(l){6-7}
$N$ & $N_t$ & $N_\Gamma$ & $\ell^\infty$ error & Order & $\ell^\infty$ error & Order \\
\midrule
64  & 100  & 160 & 3.52e-04 & ---  & 2.56e-03 & ---  \\
128 & 400  & 320 & 6.11e-05 & 2.53 & 6.57e-04 & 1.96 \\
256 & 1600 & 640 & 2.26e-05 & 1.43 & 1.58e-04 & 2.05 \\
\bottomrule
\end{tabular}
\end{table}

\begin{table}[htbp]
\centering
\caption{Self-convergence of the interface
concentration at $t=3$.}
\label{tab:coupled-interface-t3}
\begin{tabular}{@{}c c c c c c c@{}}
\toprule
 & & & \multicolumn{2}{c}{$c_{\rm i}$} & \multicolumn{2}{c}{$c_{\rm e}$} \\
\cmidrule(lr){4-5}\cmidrule(l){6-7}
$N$ & $N_t$ & $N_\Gamma$ & $\ell^\infty$ error & Order & $\ell^\infty$ error & Order \\
\midrule
64  & 100  & 160 & 4.53e-04 & ---  & 2.26e-03 & ---  \\
128 & 400  & 320 & 9.76e-05 & 2.21 & 5.72e-04 & 1.98 \\
256 & 1600 & 640 & 1.90e-05 & 2.36 & 1.41e-04 & 2.02 \\
\bottomrule
\end{tabular}
\end{table}

\begin{figure}[htbp]
    \centering
    \includegraphics[width=0.80\textwidth]{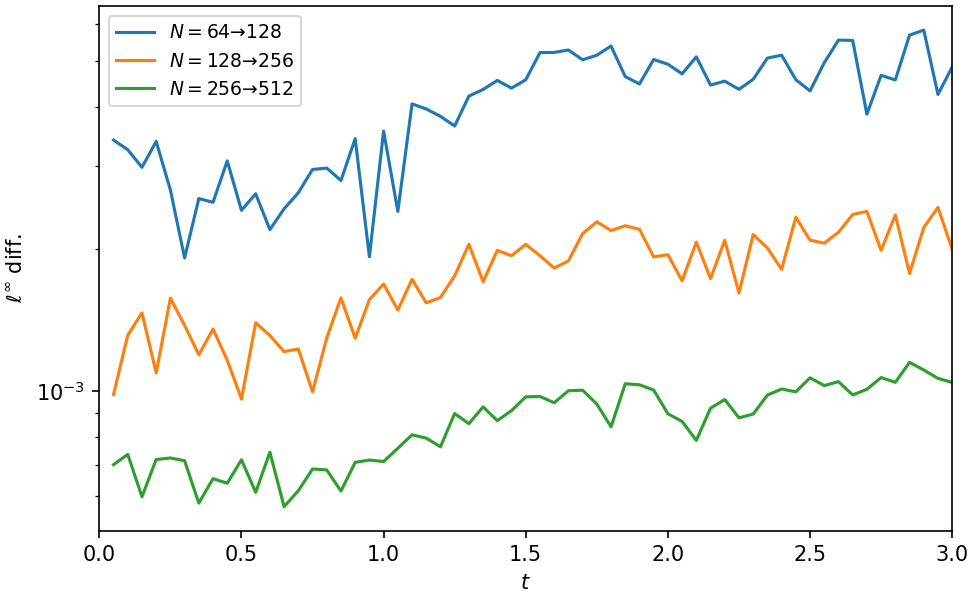}
    \caption{Differences between successive grid levels for
    the coupled chemical dynamics.}
    \label{fig:coupled-refinement-diff}
\end{figure}

\begin{figure}[htbp]
    \centering
    \includegraphics[width=0.90\textwidth]{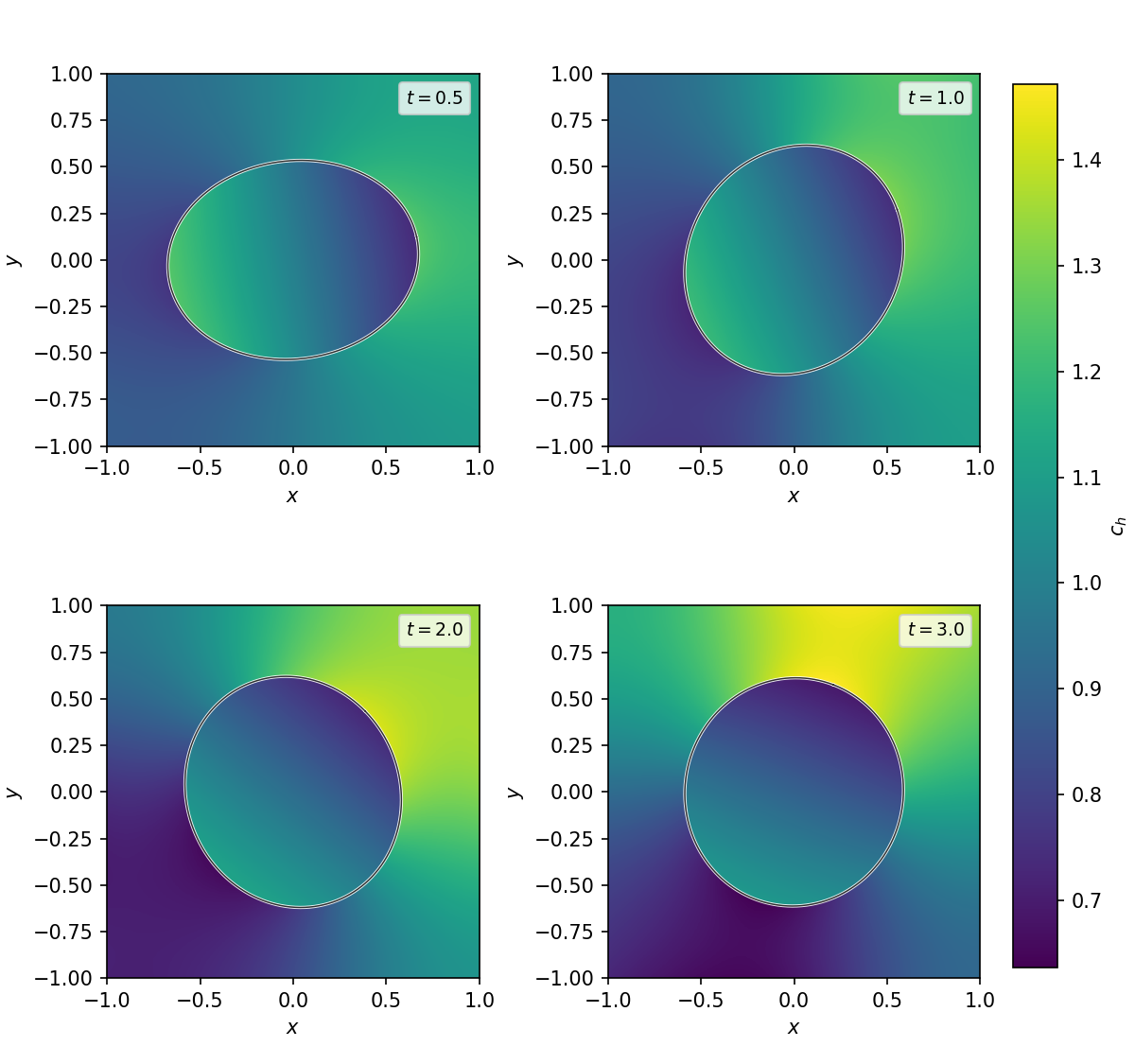}
    \caption{Snapshots of the interior and exterior
    concentrations in the coupled problem.}
    \label{fig:coupled-snapshots}
\end{figure}

\begin{figure}[htbp]
    \centering
    \includegraphics[width=0.84\textwidth]{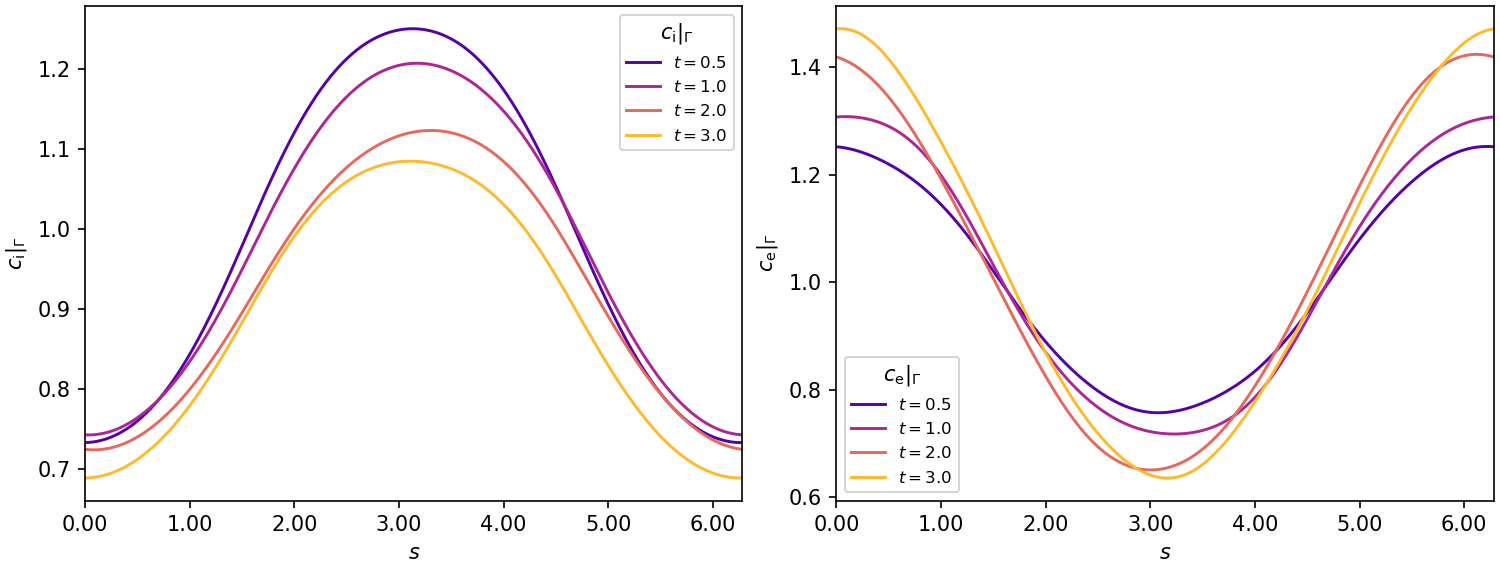}
    \caption{Interface concentration profiles in the coupled
    problem.}
    \label{fig:coupled-interface-profiles}
\end{figure}

\subsubsection{GMRES iteration counts}
Figure~\ref{fig:coupled-gmres} reports the average number of
GMRES iterations per time step for the interior and exterior interface-density
solves. The counts are nearly independent of the grid size. This behavior again confirms the linear-complexity discussion in Section~\ref{sec:method}.

\begin{figure}[htbp]
    \centering
    \includegraphics[width=0.84\textwidth]{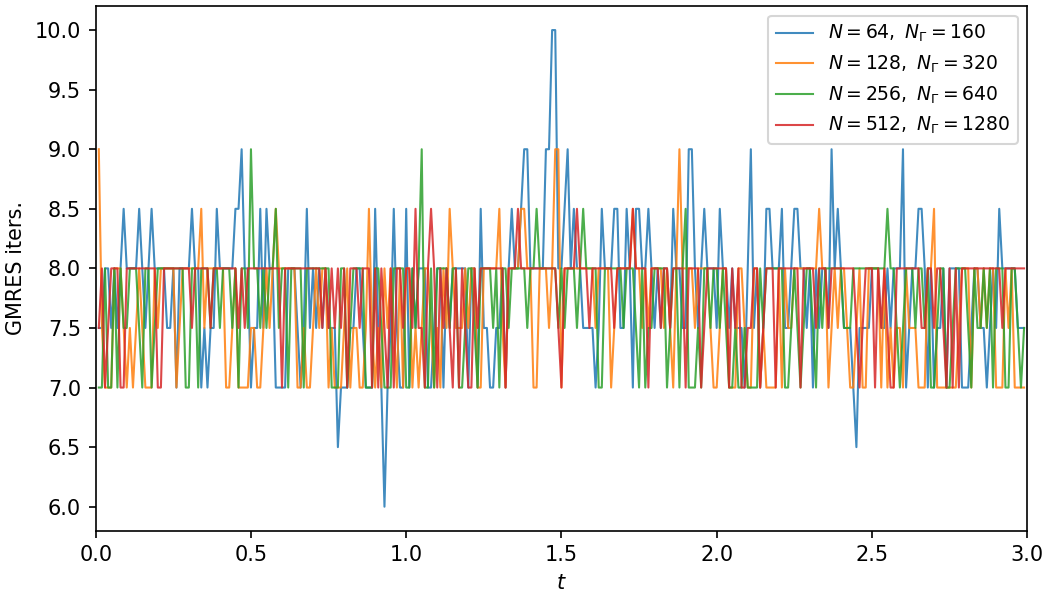}
    \caption{Average GMRES iterations per time step for the
    interior and exterior solves in the coupled problem.}
    \label{fig:coupled-gmres}
\end{figure}

\subsubsection{Conservation study}
Due to no-flux boundary conditions and the flux balance on the interface, the total mass in the system should be conserved. The numerical method does not enforce this property exactly, but we expect the mass conservation error to decrease under refinement.
We monitor conservation using the composite midpoint quadrature rule, evaluated with the same cell-centered values used in the discretization.
Let
\begin{equation}
\begin{aligned}
\mc I_{\rm i}^n
&=\left\{(p,q):\bm x_{p+\frac12,q+\frac12}\in\Omega_{\rm i}(t^n)\right\},\\
\mc I_{\rm e}^n
&=\left\{(p,q):\bm x_{p+\frac12,q+\frac12}\in\Omega_{\rm e}(t^n)\right\}.
\end{aligned}
\end{equation}
The instantaneous mass error is
\begin{equation}
    C_m^n
    = h^2 \sum_{(p,q)\in\mc I_{\rm i}^n}
    \left(c_{{\rm i},p+\frac12,q+\frac12}^n-c_0\right)
    + h^2 \sum_{(p,q)\in\mc I_{\rm e}^n}
    \left(c_{{\rm e},p+\frac12,q+\frac12}^n-c_0\right).
\end{equation}
We also use the discrete-in-time $L^2$ accumulated mass error
\begin{equation}
    E_{m,2}^n = \left(\tau\sum_{\ell=1}^n \abs{C_m^\ell}^2\right)^{1/2}.
\end{equation}
Figure~\ref{fig:conservation} shows both quantities. The instantaneous mass error remains small, and the accumulated error decreases under refinement.

\begin{figure}[htbp]
    \centering
    \includegraphics[width=0.9\textwidth]{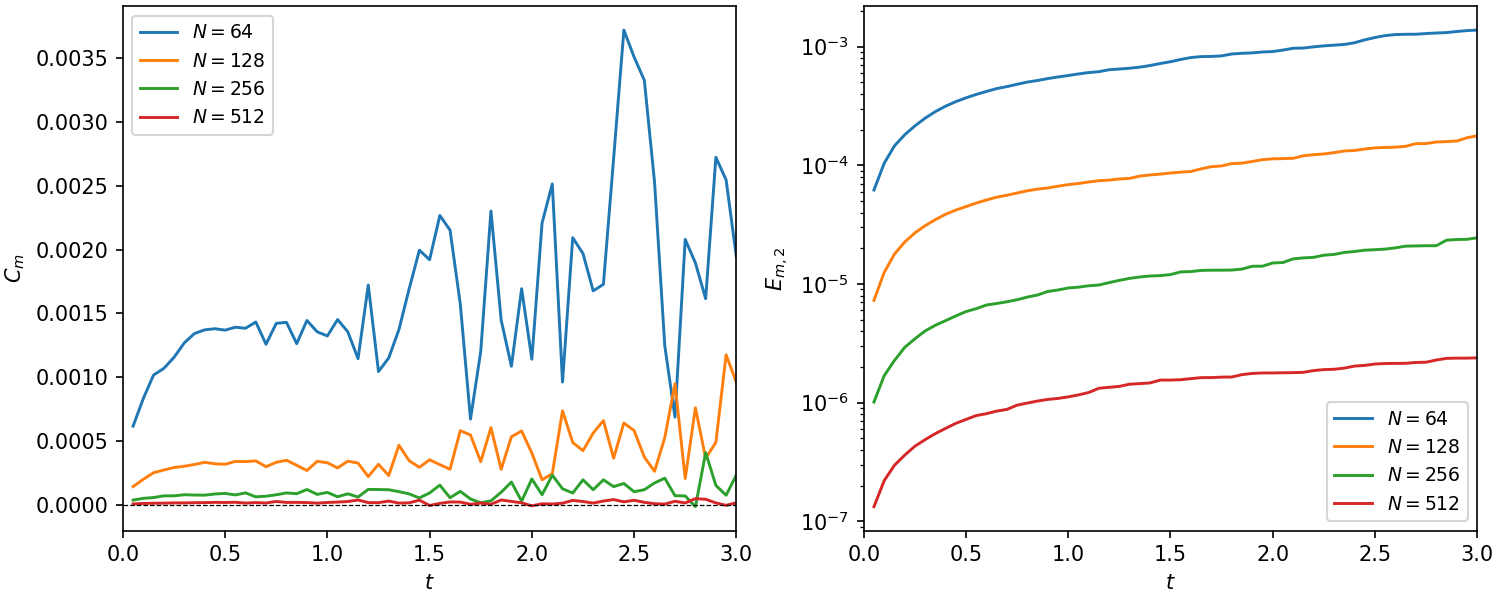}
    \caption{Conservation diagnostics for the coupled
    problem. Left: instantaneous mass error $C_m^n$. Right: accumulated mass
    error $E_{m,2}^n$.}
    \label{fig:conservation}
\end{figure}

\section{Conclusion}\label{sec:conclusion}

This paper developed a correction-based Cartesian grid method for 
advection--diffusion equations with Robin boundary conditions on moving
domains. The moving-domain problem is recast as a moving interface problem on a fixed
box, with an interface density used to impose the Robin condition. A local polynomial correction function is used to provide the jump information and correction terms at irregular grid points near the moving boundary. Because these corrections enter the scheme
only through the right-hand side, the bulk operator remains the standard
cell-centered Cartesian finite-difference operator and can be solved efficiently.
The formulation keeps the computational geometry modest. It does not require cut-cell volumes, quadrature over partial cells, explicit grid--interface intersection reconstruction, or small-cell stabilization. After the local correction coefficients and the bulk unknowns are eliminated, the remaining interface-density system is solved by matrix-free GMRES, with each
matrix-vector product requiring one multigrid solve for the regular Cartesian bulk operator.

The analysis proves $\mc O(\tau+h^2)$ error estimate for the fully discrete one-dimensional scheme under the stated assumptions. The numerical experiments are consistent with
this estimate: the one-dimensional manufactured solution shows first-order
convergence in time and second-order convergence in space. The two-dimensional
manufactured examples show comparable accuracy for the bulk concentration and
the interface trace, and the GMRES iteration counts remain nearly independent
of the mesh size. In the active-flux example, successive refinements show
self-convergence for coupled interior and exterior concentrations, while the
conservation diagnostics show small instantaneous mass defects and decreasing
accumulated mass error.

The principal remaining analytical issue is the extension of the convergence
theory to multiple space dimensions, which is significantly more difficult due to the complex geometry of the moving interface and grid-interface intersections. Another natural direction is to treat coupled motion laws, where the interface velocity is determined by the computed concentration rather than prescribed or updated separately.

\section*{Acknowledgments}
The authors gratefully acknowledge support from the National Institute for Theory and Mathematics in Biology (NITMB).
YM was partially supported by NSF DMR-2309034 (MRSEC) and the Simons Foundation Math+X Chair Fund MPS-MATHX-00234606.

\section*{Conflicts of Interest}
The authors declare no conflicts of interest.

\section*{Data Availability Statement}
Data available on request from the authors.

\bibliographystyle{amsplain}
\bibliography{references}
\end{document}